\documentclass[11pt,letterpaper]{amsart}

 \usepackage{txfonts} %%%%%%%%瀛椾綋 $\lambda$
\usepackage{graphicx}
\usepackage{verbatim}
\usepackage{tikz}

\usepackage{caption}
\usepackage{subcaption}

\usepackage[colorlinks,linkcolor=blue,citecolor=blue,urlcolor=blue]{hyperref}

 \theoremstyle{plain}

\newtheorem{theorem}{Theorem}[section]
\newtheorem{lemma}[theorem]{Lemma}

\newtheorem{corollary}[theorem]{Corollary}
\newtheorem{proposition}[theorem]{Proposition}

\newtheorem{conjecture}[theorem]{Conjecture}

\theoremstyle{definition}
\newtheorem{definition}[theorem]{Definition}
\newtheorem{remark}[theorem]{Remark}
\newtheorem{example}[theorem]{Example}
\numberwithin{equation}{section}

\makeatletter
\@namedef{subjclassname@1991}{2020 Mathematics Subject Classification}
\makeatother
\begin{document}

\captionsetup[figure]{labelfont={bf},labelformat={default},labelsep=period,name={Fig.}}

\title[Morse theory and moduli spaces of simple polylines and polygons]{Morse theory and moduli spaces of simple polylines and polygons}

\author{Te Ba}

\address{Beijing International Center for Mathematical Research, Peking University, Beijing 100871, China}
\address{School of Mathematics, Hunan University, Changsha 410082, China}
\email{batexu@hnu.edu.cn}
\author{Ze Zhou}
\address{School of Mathematical Sciences, Shenzhen University,  Shenzhen 518060, China }
\email{zhouze@szu.edu.cn}
%\thanks{The author is supported by NSFC (No.12371075 and No.11631010).}
%\textsc{Ze Zhou} \\
%\normalsize School of Mathematical Sciences, Shenzhen University,  Shenzhen 518060, China \\
%\normalsize{zhouze@szu.edu.cn}
%}

\date{}

\subjclass{Primary 52C25, 57R70; Secondary 68U05.}%55U05,
%, 68T40 70B15, 55R80

\begin{abstract}
We prove that a smooth $d$-manifold $M$ is diffeomorphic to $\mathbb R^d$ if it admits a Lyapunov-Reeb function, i.e., a smooth map $f\colon M\to\mathbb R$ that is proper, lower-bounded, and possesses exactly one critical point. We further establish that the moduli spaces of simple polylines and positively oriented polygons with fixed or varying side lengths are diffeomorphic to Euclidean spaces, by explicitly constructing Lyapunov-Reeb functions on these spaces. This yields a refined resolution to the Carpenter's Rule Problem and settles a conjecture proposed by Gonz\'{a}lez and Sedano-Mendoza. In addition, we describe foliation structures of the moduli spaces and develop algorithms for the refolding problem.
\end{abstract}

\maketitle
%\tableofcontents

\section{Introduction}\label{S-1}
Let $G(V,E)$ be a simple planar graph and let $\ell:E\to\mathbb{R}_+$ be an assignment of positive reals to the edges of $G$. A \textbf{configuration} for $G$ in $\mathbb R^2$ is a map $p:G\to\mathbb R^{2}$ such that the image $p(e)$ is a line segment for each $e\in E$, and a \textbf{linkage} for $(G,\ell)$ in $\mathbb R^2$ is a configuration $p: G\to \mathbb R^{2}$ such that $p(e)$ has length $\ell(e)$ for every $e\in E$.  Over the past decades, linkage theory has emerged as a foundational framework across diverse disciplines, including discrete and computational geometry~\cite{Alt-2003}, topological robotics~\cite{Lumelsky-1987}, molecular biology~\cite{Aichholzer-2003}, polymer physics~\cite{Frank-Kamenetskii-1997}, and related fields. For additional background, we refer readers to the seminal monographs by Demaine and O'Rourke~\cite{Demaine-2007}, and Connelly and Demaine~\cite{Connelly-2009}.

%Farber~\cite{Farber-2008},

We say a configuration $p: G\to \mathbb R^{2}$ is \emph{self-avoiding} if it is an embedding. The moduli space $\mathcal M(G)$ is defined to be the set of equivalence classes of all self-avoiding configurations for $G$ in $\mathbb R^2$, where two configurations $p_1,p_2: G\to \mathbb R^{2}$ are equivalent if there exists an orientation-preserving isometry
$\zeta:\mathbb R^{2}\to \mathbb R^{2}$ such that $p_2=\zeta\circ p_1$.  For $p_1,p_2\in\mathcal M(G)$, the distance between them is defined as
\[
d(p_1,p_2)=\inf\nolimits_{\zeta}\left(\sum\nolimits_{v\in V}\|p_2(v)-\zeta\circ p_1(v)\|\right),
\]
where the  infimum is taken over all orientation-preserving isometries of $\mathbb R^2$. Similarly, we define $\mathcal M(G,\ell)\subset \mathcal M(G)$ as the subset of equivalence classes of self-avoiding linkages for $(G,\ell)$ in $\mathbb R^2$ and endow it with the induced metric.

In this paper we restrict our attention to polygonal linkages and configurations. Specifically, we assume $G=R_m$ with $m\geq2$ or $G=C_m$ with $m\geq3$, where $R_m$ (resp. $C_m$) denotes the graph with vertices $v_1,\cdots, v_m$ such that an edge exists between $v_i$ and $v_j$ if and only if
\[
|j-i|=1\quad \left(\text{resp}.\;\;|j-i|=1\;\;\text{or}\;\;|j-i|=m-1\right).
\]
Up to orientation-preserving isometries, the moduli space $\mathcal M(R_m)$
(resp. $\mathcal M(C_m)$) parameterizes all simple polylines (resp. polygons) in
$\mathbb R^2$ with $m$ marked vertices. Accordingly, the subspace
$\mathcal M(R_m,\ell)\subset \mathcal M(R_m)$
(resp. $\mathcal M(C_m,\ell)\subset \mathcal M(C_m)$) parameterizes simple polylines (resp. polygons) with side lengths prescribed by $\ell$. We call an element in $\mathcal M(R_m)$ (resp. $\mathcal M(R_m,\ell)$) an \textbf{arm configuration} (resp. \textbf{arm linkage}) and call an element in $\mathcal M(C_m)$ (resp. $\mathcal M(C_m,\ell)$) a \textbf{cycle configuration} (resp. \textbf{cycle linkage}), as illustrated in Fig.~\ref{F-1}.

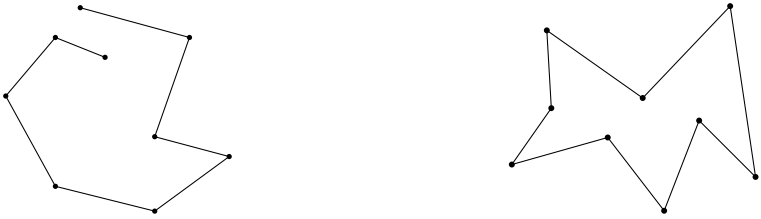
\begin{figure}[htbp]
     \centering
     \begin{minipage}[b]{0.40\textwidth}
     \centering
\begin{tikzpicture}[scale=0.55]
\coordinate (a) at (-0.5,-0.4);
\coordinate (b) at (1.7,-1);
\coordinate (c) at (1,-3);
\coordinate (d) at (2.5,-3.4);
\coordinate (e) at (1,-4.5);
\coordinate (f) at (-1,-4);
\coordinate (g) at (-2,-2.18);
\coordinate (h) at (-1,-1);
\coordinate (i) at (0,-1.4);
\draw (a)--(b)--(c)--(d)--(e)--(f)--(g)--(h)--(i);%--(j)
\fill (a) circle (1.5pt);
\fill (b) circle (1.5pt);
\fill (c) circle (1.5pt);
\fill (d) circle (1.5pt);
\fill (e) circle (1.5pt);
\fill (f) circle (1.5pt);
\fill (g) circle (1.5pt);
\fill (h) circle (1.5pt);
\fill (i) circle (1.5pt);
\end{tikzpicture}
     \end{minipage}
     \hfill
     \begin{minipage}[b]{0.50\textwidth}
     \centering
\begin{tikzpicture}[scale=0.55]
\coordinate (a) at (-0.7,0.2);
\coordinate (b) at (1,-1);
\coordinate (c) at (2.55,0.63);
\coordinate (d) at (3,-2.4);
\coordinate (e) at (2,-1.4);
\coordinate (f) at (1.38,-3);
\coordinate (g) at (0.38,-1.7);
\coordinate (h) at (-1.32,-2.18);
\coordinate (i) at (-0.62,-1.18);

\draw (a)--(b)--(c)--(d)--(e)--(f)--(g)--(h)--(i)--(a);
\fill (a) circle (1.5pt);
\fill (b) circle (1.5pt);
\fill (c) circle (1.5pt);
\fill (d) circle (1.5pt);
\fill (e) circle (1.5pt);
\fill (f) circle (1.5pt);
\fill (g) circle (1.5pt);
\fill (h) circle (1.5pt);
\fill (i) circle (1.5pt);
\end{tikzpicture}
     \end{minipage}
     \hfill
\caption{An arm linkage and a cycle linkage}\label{F-1}
\end{figure}

A fundamental problem in linkage theory is to investigate the topology of moduli spaces. For instance, the connectivity of $\mathcal M(R_m,\ell)$ or $\mathcal M(C_m,\ell)$ is closely tied to the Carpenter's Rule Problem: \emph{Is it always possible to deform any simple polyline to a line segment} (by G. Bergman), \emph{or to deform any simple polygon to a convex polygon} (by S. Schanuel), \emph{such that no self-intersection is created and each side remains a fixed length during the deformation process?} It was mentioned in Kirby's list of \textsl{Problems in Low-Dimensional Topology} (see~\cite[Problem 5.18]{Kirby-1997}) and remained open for a long time until Connelly, Demaine and Rote~\cite{Connelly-2003} provided a celebrated answer in 2003. An alternative approach, along with an efficient algorithm, was developed by Streinu~\cite{Streinu-2005}. Additionally, Lenhart and Whitesides~\cite{Lenhart-1995} and Aichholzer et al.~\cite{Aichholzer-2001} described the topology of subspaces of convex cycle linkages. The results combined lead to the following theorem.

\begin{theorem}\label{T-1-1} Let $R_m$ and $C_m$ be defined as above. The following hold:
\begin{itemize}
\item[$(i)$] $\mathcal M(R_m,\ell)$ is path-connected.
\item[$(ii)$] $\mathcal M(C_m,\ell)$ is either empty or consists of exactly two connected components, where each linkage in one component is the reflection of a linkage in the other component.
\end{itemize}
\end{theorem}

\begin{remark}\label{R-1-2}
By a theorem of Penner~\cite[Theorem 6.2]{Penner-1987}, $\mathcal M(C_m,\ell)\neq \emptyset$ if and only if $\ell:E\to \mathbb R_{+}$ satisfies the following condition:
\begin{itemize}
\item[$\mathrm{\mathbf{(c1)}}$]For every $e_i\in E$, $\ell(e_i)<\sum_{e_j\in E\setminus\{e_i\}} \ell(e_j)$.
\end{itemize}
\end{remark}

%Farber's monograph~\cite{Farber-2008}

It is of interest to explore topological properties of these moduli spaces in greater depth. Let $\mathcal M^{+}(C_m,\ell)\subset \mathcal M(C_m,\ell)$ be the component of positively oriented cycle linkages. Connelly, Demaine and Rote~\cite{Connelly-2003} claimed $\mathcal M(R_m,\ell)$ and $\mathcal M^{+}(C_m,\ell)$ are contractible, with proofs later provided in research of Han et al.~\cite{Han-2013} and Shimamoto and Wootters~\cite{Shimamoto-2014}. Actually, Shimamoto and Wootters~\cite{Shimamoto-2014} further deduced that $\mathcal M^{+}(C_m,\ell)$ is homeomorphic to $\mathbb R^{m-3}$
under condition~$\mathrm{\mathbf{(c1)}}$ and the constraint that $\sum_{e_i\in E}\sigma(e_i)\ell(e_i)\neq0$ for every sign assignment $\sigma: E\to\{-1,1\}$ to the edges.

These results motivate the Refined Carpenter's Rule Problem for cycle linkages: \emph{Is $\mathcal M^{+}(C_m,\ell)$ homeomorphic to $\mathbb R^{m-3}$ whenever $\mathcal M^{+}(C_m,\ell)\neq \emptyset$?} This problem is tantamount to asking whether the constraint stipulated in Shimamoto and Wootters' result can be removed. The core challenge arises when a sequence of linkages in $\mathcal M^{+}(C_m,\ell)$ approaches a pair of overlapping line segments (see Remark~\ref{R-4-13}).
An analogous question can also be posed for the moduli space $\mathcal M(R_m,\ell)$.

Recently, characterizing the topology of the moduli space $\mathcal M(C_m)$ has garnered significant attention. In~\cite{Gonzalez-2018}, Gonz$\acute{\text{a}}$lez proved that $\mathcal M(C_m)$ consists of two homeomorphic components, each of which is simply connected. Let $\mathcal M^{+}(C_m)\subset \mathcal M (C_m)$ be the component corresponding to positively oriented polygons. Subsequently, Gonz$\acute{\text{a}}$lez and Sedano-Mendoza~\cite[Conjecture 1.1]{Gonzalez-2023} proposed the following conjecture: \emph{$\mathcal M^{+}(C_m)$ is homeomorphic to $\mathbb R^{2m-3}$}. For $m=4$, this
conjecture was validated by Gonz$\acute{\text{a}}$lez and  L$\acute{\text{o}}$pez-L$\acute{\text{o}}$pez~\cite{Gonzalez-2016}. Thereafter, Fortier Bourque~\cite{Bourque-2024} investigated the moduli spaces of marked immersed polygons and solved the problem for $m=5$. To date, the general case for $m>5$ remains open.

This paper aims to address the aforementioned problems. We first note that the moduli spaces $\mathcal M(R_m)$, $\mathcal M(R_m,\ell)$, $\mathcal M^{+}(C_m)$, and $\mathcal M^{+}(C_m,\ell)$ are smooth (i.e., $C^\infty$) manifolds (see Lemmas~\ref{L-3-1} and~\ref{L-4-1}). Leveraging smooth structures, we shall give a complete description of the topology of these moduli spaces. Throughout, the Euclidean $d$-space $\mathbb R^d$ and the $d$-sphere $S^d$ are assumed to carry their standard smooth structures.

\begin{theorem}\label{T-1-3}
Suppose $m\geq 2$ and suppose $\ell:E\to\mathbb R_{+}$ is an arbitrary assignment. We have the following conclusions:
\begin{itemize}
\item[$(i)$] $\mathcal M(R_m,\ell)$ is diffeomorphic to $\mathbb{R}^{m-2}$.
\item[$(ii)$]$\mathcal M(R_m)$ is diffeomorphic to $\mathbb{R}^{2m-3}$.
\end{itemize}
\end{theorem}

\begin{theorem}\label{T-1-4}%\rm{\textbf{(c1)}}
Suppose $m\geq 3$ and suppose $\ell:E\to\mathbb R_{+}$ is an assignment satisfying condition $\mathrm{\mathbf{(c1)}}$
(in Remark~\ref{R-1-2}). The following statements hold:
\begin{itemize}
\item[$(i)$]  $\mathcal M^{+}(C_m,\ell)$ is diffeomorphic to $\mathbb{R}^{m-3}$.
\item[$(ii)$] $\mathcal M^{+}(C_m)$ is diffeomorphic to $\mathbb{R}^{2m-3}$.
\end{itemize}
\end{theorem}

\begin{remark}
Theorems~\ref{T-1-3} and~\ref{T-1-4} settle the Refined Carpenter's Rule Problem and offer a unified generalization of the results of Connelly, Demaine and Rote~\cite{Connelly-2003} and Shimamoto and Wootters~\cite{Shimamoto-2014}. In addition, part $(ii)$ of Theorem~\ref{T-1-4} confirms the conjecture of Gonz\'{a}lez and Sedano-Mendoza~\cite[Conjecture 1.1]{Gonzalez-2023}.
\end{remark}

%Kamiyama~\cite{Kamiyama-1998}, Milgram and Trinkle~\cite{Milgram-2004}Khimshiashvilli~\cite{Khimshiashvili-2009}, Sch\"{u}tz~\cite{Schuetz-2010,Schuetz-2013,Schuetz-2016}, Farber, Hausmann and Sch\"{u}tz~\cite{Farber-Hausmann-2011}, Farber and Fromm~\cite{Farber-Fromm-2012},Farber and Sch\"{u}tz~\cite{Farber-2007}, Hausmann~\cite{Hausmann-1989},,Schuetz-2013,Schuetz-2016
%Sch\"{u}tz~\cite{Schuetz-2010},O'Hara~\cite{Ohara-1998,Ohara-2013},

To demonstrate the above results, we shall make use of Morse theory, which has motivated extensive research on linkages allowing self-intersections. Representative studies include those by Kapovich and Millson~\cite{Kapovich-1995}, Kamiyama~\cite{Kamiyama-1998,Kamiyama-2014,Kamiyama-2019}, O'Hara~\cite{Ohara-1998,Ohara-2013}, Shimamoto and Vanderwaart~\cite{Shimamoto-2005}, Sch\"{u}tz~\cite{Schuetz-2010}, Farber and Fromm~\cite{Farber-Fromm-2012}, among others. In addition, we mention the relevant work of Luo~\cite{Luo-2022}, as well as Luo, Wu and Zhu~\cite{Luo-Wu-2023,Luo-Wu-Zhu-2023,Luo-Wu-Zhu-2024}, concerning moduli spaces of geodesic triangulations.

However, to adapt to the study of self-avoiding linkages and configurations, we need to go a step further to develop Morse theory for open manifolds. To start with, we assume $M$ is a smooth manifold.

\begin{definition}\label{D-1-6}
A smooth map $f:M\to \mathbb R$ is called a \textbf{Lyapunov-Reeb function} (or \textbf{L-R function} for short) on $M$ if it is \emph{proper}, \emph{lower-bounded}, and \emph{has exactly one critical point}.
\end{definition}

We recall that a \emph{proper map} between two topological spaces means the preimage of every compact subset is compact. The first contribution of this paper is establishing the following analogue of the Reeb Sphere Theorem~\cite{Reeb-1946,Milnor-1964}.

\begin{theorem}\label{T-1-7}
Let $M$ be a smooth manifold of dimension $d\geq 1$ that admits a Lyapunov-Reeb function $f$. Then the following properties hold:
\begin{itemize}
\item[$(i)$] Given a Riemannian metric on $M$, the negative gradient flow of $f$ starting at any point in $M$ exists for all $t\geq 0$ and converges to the critical point of $f$ as $t\to+\infty$.
\item[$(ii)$] $M$ is diffeomorphic to $\mathbb R^d$.
\item[$(iii)$] For any $c>\inf f$, the strict sublevel set $M_c=\big\{x\in M: f(x)<c\big\}$ is diffeomorphic to $\mathbb R^d$, and the level set $L_c=\big\{x\in M: f(x)=c\big\}$ is homeomorphic to $S^{d-1}$. Moreover, $L_c$ is diffeomorphic to $S^{d-1}$ provided $d\neq 5$ or the critical point of $f$ is nondegenerate.
\end{itemize}
\end{theorem}

\begin{remark}
We note that Wilson~\cite{Wilson-1967} and Kamiyama~\cite{Kamiyama-2015} obtained analogous results to statement $(iii)$ under alternative assumptions. Their results do not imply our conclusion, as the flow map in our setting may fail to be invertible.
\end{remark}

%The definition of L-R functions is motivated by the two relations: $(1)$ Up to a constant, an L-R function serves as a Lyapunov function for its negative gradient flow; $(2)$ As in the Reeb Sphere Theorem~\cite{Milnor-1964}, the critical point of an L-R function is allowed to be degenerate.

In order to show Theorems~\ref{T-1-3} and~\ref{T-1-4}, it suffices to construct L-R functions on the relevant spaces. We first introduce some concepts motivated by the work of Cantarella, Demaine, Iben and O'Brien~\cite{Cantarella-2004}.

\begin{definition}\label{D-1-9}
For $U=\mathcal M(R_m,\ell)$ or $U=\mathcal M^{+}(C_m,\ell)$,  a \emph{proper} smooth map $\Phi:U\to[0,+\infty)$ is called a \textbf{strain energy} on $U$ if it satisfies
\[
\frac{\mathrm{d}\Phi(p_t)}{\mathrm{d}t}<0
\]
for every expansive motion $\{p_t\}_{-\varepsilon<t<\varepsilon}$ ($\varepsilon>0$) in $U$ and every $t\in(-\varepsilon,\varepsilon)$.
\end{definition}

Here we say a smooth family $\{p_t\}_{-\varepsilon<t<\varepsilon}$ ($\varepsilon>0$) of linkages in $U$ is an \emph{expansive motion} in $U$  if, for all $t\in (-\varepsilon,\varepsilon)$ and
all pairs of vertices $v_i,v_j\in V$,
\[
\frac{\mathrm{d}}{\mathrm{d}t}
~\| p_{t}(v_i)-p_{t}(v_j)\|\geq 0,
\]
with at least one inequality strict.

\begin{remark}\label{R-1-10}
Let $m\geq 3$. It is easy to see that the function
\[
\Phi(p)=\sum_{\substack{[v_i,v_j]\in E \\ v_k\in V\setminus\{v_i,v_j\}}}\frac{1}
{\left(\|p(v_i)-p(v_k)\|+\|p(v_j)-p(v_k)\|-\|p(v_i)-p(v_j)\|\right)^2}
\]
introduced by Cantarella et al.~\cite{Cantarella-2004} is a strain energy on
$\mathcal M(R_m,\ell)$ (resp. $\mathcal M^{+}(C_m,\ell)$).
\end{remark}

We observe that strain energies on $\mathcal M(R_m,\ell)$ serve as L-R functions, and we may render the critical point nondegenerate by adding the chord energy. Here the \emph{chord energy} of an arm linkage $p\in \mathcal M(R_m,\ell)$ is given by
\[
\varpi(p):=\frac{1}{\|p(v_1)-p(v_m)\|^2}.
\]
\begin{theorem}\label{T-1-11}
Let $m\geq 3$ and let $\ell:E\to\mathbb R_{+}$ be an assignment. Then every strain energy $\Phi:\mathcal M(R_m,\ell)\to [0,+\infty)$ gives a Lyapunov-Reeb function on $\mathcal M(R_m,\ell)$. Moreover, $f=\Phi+\varpi$ is a Lyapunov-Reeb function on $\mathcal M(R_m,\ell)$ with a nondegenerate critical point.
\end{theorem}

Next we construct L-R functions on $\mathcal M^{+}(C_m,\ell)$. To this end, we introduce two functions on $\mathcal M^{+}(C_m)$ (and hence on $\mathcal M^{+}(C_m,\ell)$).
Recall that each $p\in \mathcal M^{+}(C_m)$ gives a positively oriented simple polygon $p(C_m)$. Let $A(p)$ be the area enclosed by $p(C_m)$. This gives rise to a function
\[
A: \mathcal M^{+}(C_m)\to \mathbb R_{+}.
\]
Following Shimamoto and Wootters~\cite{Shimamoto-2014}, for $i=1,2,\cdots,m$, we set
\[
w_i(p)=
\begin{cases}
\begin{aligned}
&\exp\left(\frac{1}{\pi-\alpha_i(p)}\right), &&\alpha_i(p)>\pi,\\
&0, &&\alpha_i(p)\leq\pi,
\end{aligned}
\end{cases}
\]
where $\alpha_i(p)$ denotes the interior angle of $p(C_m)$ at vertex $p(v_i)$. We then define a function $w: \mathcal M^{+}(C_m)\to [0,1)$ as
\[
w=\frac{1}{m}\left(\sum\nolimits_{i=1}^m w_i\right).
\]
Note that $w(p)=0$ if and only if $p\in \mathcal M^{+}(C_m)$ corresponds to a convex polygon. We acquire L-R functions on $\mathcal M^{+}(C_m,\ell)$ as follows.

\begin{theorem}\label{T-1-12}
Let $m\geq 4$ and let $\ell:E\to\mathbb R_{+}$ satisfy condition~$\mathrm{\mathbf{(c1)}}$. Suppose $\Phi:\mathcal M^{+}(C_m,\ell)\to [0,+\infty)$ is a strain energy. Then
\[
f=\frac{1}{A}+w\Phi
\]
is a Lyapunov-Reeb function on $\mathcal M^{+}(C_m,\ell)$ whose critical point is nondegenerate.
\end{theorem}

The proofs of the above two theorems rely on a fundamental result established by Connelly, Demaine, and Rote~\cite{Connelly-2003} (see also the work of Rote, Santos, and Streinu~\cite{Rote-2003}), which states that every non-straight arm linkage and every non-convex cycle linkage admits expansive motions. In fact, this property implies that strain energies have no critical points in certain subsets of moduli spaces.

By incorporating a projection method, we may further construct L-R functions on $\mathcal M(R_m)$ and $\mathcal M^{+}(C_m)$ (see Theorems~\ref{T-3-6} and~\ref{T-4-16}). Combining these arguments with Theorem~\ref{T-1-7} then derives Theorems~\ref{T-1-3} and~\ref{T-1-4}. In fact, this framework also yields several additional corollaries. Here we present two results which describe foliation structures of the moduli spaces via level sets of L-R functions.

\begin{corollary}\label{C-1-13}
Let $f$ be a Lyapunov-Reeb function on $\mathcal M(R_m,\ell)$ with $m\geq 3$. Given $c>\inf f$, $f^{-1}(c)$ is homeomorphic to $S^{m-3}$. In addition, $f^{-1}(c)$ is diffeomorphic to $S^{m-3}$ provided either $m\neq 7$ or the critical point of $f$ is nondegenerate.
\end{corollary}

\begin{corollary}\label{C-1-14}
Suppose $\ell:E\to\mathbb R_{+}$ satisfies condition~$\mathrm{\mathbf{(c1)}}$ and
suppose $f$ is a Lyapunov-Reeb function on $\mathcal M^{+}(C_m,\ell)$ with $m\geq4$. Given $c>\inf f$, $f^{-1}(c)$ is homeomorphic to $S^{m-4}$. Furthermore, $f^{-1}(c)$ is diffeomorphic to $S^{m-4}$ provided either $m\neq 8$ or the critical point of $f$ is nondegenerate.
\end{corollary}

\begin{remark}
Corollaries~\ref{C-1-13} and~\ref{C-1-14} settle the following strengthened version of the Carpenter's Rule Problem: \emph{Under what conditions can a polygonal linkage be continuously deformed to any other linkage in the same level set of an L-R function, such that during the deformation process no self-intersection is created and the L-R function remains constant?}
\end{remark}

Other applications of L-R functions include developing algorithms for the original Carpenter's Rule Problem through associated negative gradient flows. Under such evolution, any simple polygon can be deformed into a polygon inscribed in a circle while preserving edge lengths and avoiding self-intersections. This then provides a satisfactory discrete analogue to the results on \emph{curve shortening flow} obtained by Gage and Hamilton~\cite{Gage-1986} and Grayson~\cite{Grayson-1987}. For more details on geometric flows of polygons, see the article of Chow and Glickenstein~\cite{Chow-Glickenstein-2007}. Meanwhile, L-R functions illuminate a broader question~\cite{Cantarella-2004}: \emph{How can one generate an animation from two images?} Specifically, we will construct suitable metrics on these moduli spaces and develop algorithms for the refolding problem: \emph{Given two polygonal linkages (configurations), construct an ``optimal'' motion connecting them.}

\begin{theorem}\label{T-1-16}
There exist complete Riemannian metrics (with explicit expressions) on $\mathcal M(R_m)$, $\mathcal M^{+}(C_m)$, $\mathcal M(R_m,\ell)$, and $\mathcal M^{+}(C_m,\ell)$. As a consequence, any two distinct points in each of these moduli spaces can be joined by a geodesic.
\end{theorem}

%Further details can be found in~\S\ref{S-5}.

The paper is organized as follows: In~\S\ref{S-2} we develop Morse theory for open manifolds and prove Theorem~\ref{T-1-7}. In~\S\ref{S-3} we construct L-R functions on the moduli spaces of arm linkages and configurations, and establish Theorems~\ref{T-1-11},~\ref{T-1-3} and Corollary~\ref{C-1-13}. In~\S\ref{S-4} we study the moduli spaces of cycle linkages and configurations, deducing
Theorems~\ref{T-1-12},~\ref{T-1-4} and Corollary~\ref{C-1-14}. In~\S\ref{S-5} we  construct complete Riemannian metrics on the moduli spaces and develop algorithms addressing the refolding problem. In~\S\ref{S-6} we propose several problems for further investigation.

Throughout this paper, a manifold is assumed to be a second-countable Hausdorff space that is locally homeomorphic to Euclidean space, following the standard terminology.

\section{Morse theory in the setting of open manifolds}\label{S-2}

This section is dedicated to showing Theorem~\ref{T-1-7}. Guided by the spirit of  the proof of the Reeb Sphere Theorem~\cite{Reeb-1946,Milnor-1964}, we employ negative gradient flows along with certain modifications. In particular, we draw on Wilson's insight~\cite{Wilson-1967} that a modified negative gradient flow can induce a homotopy equivalence from the given level set to a sphere. Additionally, we shall make use of a powerful tool\textemdash the Generalized Poincar$\acute{\text{e}}$ Conjecture (see, e.g.,~\cite{Smale-1961,Freedman-1982,Perelman-2002,Perelman-2003.3,Perelman-2003.7}).

\subsection{Negative gradient flows}

First we introduce some concepts and results on negative gradient flows. For an exposition, see the textbook by Lee~\cite{Lee-2018}.

Let $M$ be a smooth manifold. A flow on $M$ is a smooth map $\varphi:\mathcal D\to M$ with the following properties:
\begin{itemize}
\item[$(1)$]$\mathcal D$ is an open subset of $\mathbb R\times M$ such that $\{0\}\times M\subset \mathcal D$;
\item[$(2)$] $\varphi(0,p)=p$;
\item[$(3)$] If $(s,p)\in\mathcal D$ and $\big(t,\varphi(s,p)\big)\in\mathcal D$ such that $(t+s,p)\in\mathcal D$, then
\[
\varphi\big(t,\varphi(s,p)\big)=\varphi(t+s,p).
\]
\end{itemize}
A flow $\varphi:\mathcal D\to M$ is said to be \textbf{forward-complete} (resp. \textbf{backward-complete}) if
\[
[0,+\infty)\times M\subset \mathcal D\quad \left(\text{resp}.\;\;(-\infty,0]\times M\subset \mathcal D\right).
\]
The flow is \textbf{complete} provided $\mathcal D=\mathbb R\times M$, meaning it is both \emph{forward-} and \emph{backward-complete}. For a complete flow $\varphi:\mathbb R\times M\to M$, note that the flow map $\varphi(t,\cdot):M\to M$ is a diffeomorphism for every $t\in \mathbb R$. A flow $\varphi:\mathcal D\to M$ is said to be generated by a smooth vector field $X$ on $M$ if it satisfies
\[
\frac{\mathrm{d}\varphi(t,p)}{\mathrm{d}t}
=X\big(\varphi(t,p)\big).
\]

We recall the Fundamental Theorem on Flows (see, e.g.,~\cite[Theorem 9.12]{Lee-2018}).

\begin{theorem}\label{T-2-1}
Every smooth vector field $X$ on $M$ generates a unique maximal flow on $M$. Moreover, the resulting flow is complete if $X$ is compactly supported.
\end{theorem}

We mention that a \emph{maximal flow} means it admits no extension to a flow on a larger domain, and a vector field $X$ is \emph{compactly supported} if its support,  given by
$\operatorname{supp} X:=\overline{\left\{p\in M: X(p)\neq 0\right\}}$, is compact.

Suppose $M$ is endowed with a Riemannian metric. The gradient of a smooth function $f$ on $M$ is the vector field $\nabla f$ such that
\[
\langle \nabla f, X\rangle= X(f)
\]
holds for all smooth vector fields $X$ on $M$. To develop Morse theory,  we need to study the negative gradient flow of $f$, i.e., the maximal flow that satisfies
\[
\frac{\mathrm{d}\varphi(t,p)}{\mathrm{d}t}
=-\nabla f\big(\varphi(t,p)\big).
\]

For open manifolds, the negative gradient may not be compactly supported, and the induced flow may fail to be complete.

\begin{example}\label{Example-1}
Define a function on $\mathbb R^2$ by $f(x,y)=x^2+y^4$. The negative gradient flow of $f$ is expressed as
\[
\varphi\big(t,(x,y)\big)=\left(e^{-2t}x,\frac{y}{(1+8ty^2)^{\frac{1}{2}}}\right).
\]
This flow is not backward-complete. Specifically, the flow map $\varphi(t,\cdot):\mathbb R^2\to \mathbb R^2$ is non-surjective for $t>0$ and not well defined for $t<0$.
\end{example}

Fortunately, forward-completeness holds for the negative gradient flows of all proper, smooth, and lower-bounded functions.

\begin{lemma}\label{L-2-3}
Let $M$ be a smooth Riemannian manifold, and let $f:M\to\mathbb R$ be a proper, smooth, and lower-bounded function. Then the negative gradient flow of $f$ is forward-complete.
\end{lemma}
\begin{proof}
In light of Theorem~\ref{T-2-1}, there exists a unique maximal flow $\varphi:\mathcal D\to M$ generated by $-\nabla f$. For each $p\in M$, let $J_p$ be the maximal interval such that
\[
J_p\times\{p\}\subset \mathcal D.
\]
It suffices to show $J_p$ has no upper bounds. Assume the contrary, then $\beta:=\sup J_p$ is finite. Define the integral curve starting at $p$ as
\[
\begin{aligned}
\gamma_p\;:\;[0,\beta)\;&\longrightarrow\; M\\
\;\;\;t\quad\;&\longmapsto\;\varphi(t,p).
\end{aligned}
\]
By the Escape Lemma~\cite[Lemma 9.19]{Lee-2018},  $\gamma_p\big([0,\beta)\big)$ cannot be contained in any compact subset of $M$. On the other hand, a direct calculation gives
\[
\frac{\mathrm{d}f\big(\gamma_p(t)\big)}{\mathrm{d}t}
=-\| \nabla f \|^2\leq 0,
\]
which yields
\[
f\big(\gamma_p(t)\big)\leq f\big(\gamma_p(0)\big)=f(p).
\]
Recall that $f$ is lower-bounded. We obtain
\[
\gamma_p\big([0,\beta)\big)\subset f^{-1}\big([a,b]\big),
\]
where $a=\min f$ and $b=f(p)$. By the properness of $f$, the preimage $f^{-1}\big([a,b]\big)$ is compact, which contradicts the fact that $\gamma_p\big([0,\beta)\big)$ cannot be contained in any compact subset of $M$. We thus derive $\beta=+\infty$ and conclude the lemma.
\end{proof}

In addition, one asserts the convergence of the flow assuming further that $f$ possesses exactly one critical point.
\begin{theorem}\label{T-2-4}
Let $M$ be a smooth Riemannian manifold, and let
$f: M\to\mathbb R$ be a Lyapunov-Reeb function with $q$ being the critical point. Suppose $\varphi:\mathcal D\to M$ is the negative gradient flow of $f$. Then for any $p\in M$,
\[
\lim_{t\to+\infty}\varphi(t,p)=q.
\]
\end{theorem}

\begin{proof}
Consider the integral curve
\[
\begin{aligned}
\gamma_p\;:\;[0,+\infty)\;&\longrightarrow\; M\\
\;\;\;t\quad\;&\longmapsto\;\varphi(t,p).
\end{aligned}
\]
As in Lemma~\ref{L-2-3}, it is easy to see  $f\big(\gamma_p(t)\big)$ is a decreasing function of $t$. Since $f$ is lower-bounded,   $\lim_{t\to+\infty}f\big(\gamma_p(t)\big)$ exists. For each positive integer $n$, Lagrange's Mean Value Theorem implies there exists $t_n\in(n,n+1)$ such that
\[
f\big(\gamma_p(n+1)\big)-f\big(\gamma_p(n)\big)
=\frac{\mathrm{d}f\big(\gamma_p(t)\big)}{\mathrm{d}t}\bigg|_{t=t_n}
=-\|\, \nabla f\big(\gamma_p(t_n)\big)\,\|^2.
\]
Because $\lim_{t\to+\infty}f\big(\gamma_p(t)\big)$ exists, we have
\[
f\big(\gamma_p(n+1)\big)-f\big(\gamma_p(n)\big)\to 0
\]
as $n\to+\infty$. Consequently,
\[
\|\,\nabla f\big(\gamma_p(t_n)\big)\,\|^2\to 0.
\]
Note that $\big\{\gamma_p(t_n)\big\}$ lies in the compact set $f^{-1}\big([a,b]\big)$,
where
\[
a=\min f=f(q)\quad\text{and}\quad b=f(p).
\]
By passing to a subsequence if necessary, we may assume $\big\{\gamma_p(t_n)\big\}$ converges to some $q^\ast\in M$. Then $\nabla f (q^\ast)=0$, which implies $q^\ast$ is a critical point of $f$. Given that $f$ has only one critical point, we deduce $q^\ast=q$, meaning
\[
\gamma_p(t_n)\to q.
\]
As a result,
\[
\lim_{t\to+\infty}f\big(\gamma_p(t)\big)
=\lim_{n\to+\infty}f\big(\gamma_p(t_n)\big)=f(q)=\min f.
\]
We claim that $\gamma_{p}(t)$ converges to $q$. Otherwise, there exists a subsequence $\big\{\gamma_{p}(s_k)\big\}$ converging to some $\tilde{q}\neq q$. It follows that
\[
\lim_{k\to+\infty}f\big(\gamma_p(s_k)\big)=f\big(\tilde{q}\big)>\min f,
\]
contradicting our earlier conclusion. Hence
$\varphi(t,p)=\gamma_p(t)\to q$ as $t\to+\infty$.
\end{proof}

%\begin{remark}\label{R-2-4}
%Set $h=f-f(q)$. Then $h$ serves as a global Lyapunov function for its negative gradient flow. By virtue of the Barba\v{s}in-Krasovski\v{i} Theorem~\cite{Barba-1954},  $q$ is a globally asymptotically stable equilibrium point of the system, implying each point is attracted to $q$ by the flow. This offers an alternative approach to Theorem~\ref{T-2-3}.
%\end{remark}

\subsection{Morse theory on open manifolds}

To streamline the arguments, we now restate assertion $(ii)$ of Theorem~\ref{T-1-7} as the following result. As mentioned before, this can be viewed as an analogue of the Reeb Sphere Theorem~\cite{Reeb-1946,Milnor-1964} in the setting of open manifolds.

\begin{theorem}\label{T-2-5}
Suppose $M$ is a smooth manifold of dimension $d\geq1$ that admits a Lyapunov-Reeb function $f$. Then $M$ is diffeomorphic to $\mathbb R^{d}$.
\end{theorem}

Before giving the proof, we introduce two preliminary lemmas. The first one is due to Brown~\cite{Brown-1961} and Stallings~\cite{Stallings-1962}. The following version together with a proof can be seen in the lecture note of Milnor~\cite[p.~168]{Milnor-1964}.
\begin{lemma}[Brown-Stallings]\label{L-2-6}
Let $M$ be a smooth manifold such that every compact subset is contained in an open subset that is diffeomorphic to $\mathbb{R}^d$. Then $M$ is diffeomorphic to $\mathbb{R}^d$.
\end{lemma}

%\begin{remark}
%Milnor's original statement~\cite[p.~168]{Milnor-1964} assumes $M$ is paracompact. This can be derived from our presupposition that manifolds are second-countable Hausdorff spaces. For a proof,  see, e.g.,~\cite[Theorem 1.15]{Lee-2018}.
%\end{remark}

The next lemma addresses the existence of smooth bump functions.
\begin{lemma}\label{L-2-7}
Given $a,b\in \mathbb R$ with $a<b$, there exists a smooth function $\eta_{a,b}:\mathbb R\to\mathbb R$ such that $\eta_{a,b}\equiv1$ on $(-\infty,a]$, $\eta_{a,b}\equiv 0$ on $[b,+\infty)$, and $0<\eta_{a,b}<1$ on $(a,b)$.
\end{lemma}
\begin{proof}
It is easy to see that the function
\[
\eta_{a,b}(x)=
\begin{cases}
\begin{aligned}
&1, &&x\leq a,\\
&\exp\left(\dfrac{x-a}{x-b}\right), &&a<x<b,\\
&0, &&x\geq b
\end{aligned}
\end{cases}
\]
satisfies the required properties.
\end{proof}

We are ready to demonstrate Theorem~\ref{T-2-5}.

\begin{proof}[\textbf{Proof of Theorem~\ref{T-2-5}}]
By Lemma~\ref{L-2-6}, it suffices to show that every compact subset $K\subset M$ is contained in an open subset $W\subset M$ that is diffeomorphic to $\mathbb{R}^d$.

For this purpose, we first equip $M$ with a Riemannian metric. Let $q$ be the unique critical point of $f$, and let $B_q\subset M$ be a relatively compact neighborhood of $q$ that is diffeomorphic to $\mathbb{R}^d$. Suppose $\varphi:\mathcal D\to M$ is the negative gradient flow of $f$. By virtue of Proposition~\ref{P-2-9} below, there exists $T>0$ such that
\[
\varphi(T,K)\subset B_q.
\]
Let $\eta_{a,b}: \mathbb R\to\mathbb R$ be given as in Lemma~\ref{L-2-7} with
\[
a=1+\max\nolimits_{p\in K} f(p)\quad\text{and}\quad b=2+\max\nolimits_{p\in K} f(p).
\]
Then $-\eta_{a,b}(f)\nabla f$ is a smooth vector field on $M$ whose support equals
\[
K_b:=\big\{p\in M: f(p)\leq b\big\}.
\]
Since $f: M\to\mathbb R$ is proper and lower-bounded,  $K_b$ is clearly a compact set. By Theorem~\ref{T-2-1}, $-\eta_{a,b}(f)\nabla f$ generates a complete flow
$\varphi_\eta:\mathbb R\times M\to M$. Thus for every
$t\in \mathbb R$, the flow map
$\varphi_\eta(t,\cdot):M \to M$ is a diffeomorphism. Note that
\[
\frac{\mathrm{d}f\big(\varphi_\eta(t,p)\big)}{\mathrm{d}t}
=-\eta_{a,b}(f)\,\| \nabla f \|^2\leq 0,
\]
and $\eta_{a,b}(f)\equiv1$ on
$K_a:=\big\{p\in M: f(p)\leq a\big\}$.
For $t\geq 0$ and $p\in K_a$, we derive
 \[
\varphi_\eta(t,p)=\varphi(t,p)\in K_a.
\]
Because $T>0$ and $K\subset K_a$, this implies
\[
\varphi_\eta(T,K)=\varphi(T,K)\subset B_q.
\]
Setting $W=\varphi_\eta(-T,B_q)$ gives the required open set, proving the assertion.
\end{proof}

\begin{remark}
As previously noted, the original flow $\varphi$ may not be complete, which impedes defining the backward flow map on $B_q$. Fortunately, this issue is remedied by the modified flow $\varphi_\eta$, allowing the required construction to proceed.
\end{remark}

The following proposition provides a uniform bound for the ``arrival time'' of the negative gradient flow when the initial points lie in a compact set.
\begin{proposition}\label{P-2-9}
There exists $T>0$ such that for all $t\geq T$,
\[
\varphi(t,K)\subset B_q.
\]
\end{proposition}

\begin{proof}
Define
\[
\mu=\min_{p\in\partial B_q} f(p)\quad\text{and}\quad M_{\mu}=\big\{p\in M: f(p)<\mu\big\}.
\]
We claim $M_{\mu}\subset B_q$. Otherwise, there exists another critical point $\tilde{q}$ of $f$ such that
\[
f(\tilde{q})=\min_{p\in M\setminus B_q} f(p)<\mu=\min_{p\in\partial B_q} f(p),
\]
contradicting the assumption that $f$ has exactly one critical point.
For every $p\in M$, Theorem~\ref{T-2-4} guarantees the existence of $t_p\geq 0$ with the property that
\[
\varphi(t_p,p)\in M_{\mu}.
\]
Since $\varphi(t_p,\cdot): M\to M$ is a smooth map and $M_{\mu}$ is an open subset, one can find a neighborhood $U_p\subset M$ of $p$ satisfying
\[
\varphi(t_p,U_p)\subset M_{\mu}.
\]
Note that $\big\{U_p\big\}_{p\in K}$ forms an open cover of the compact set $K$ and $f\big(\varphi(t,p)\big)$ is a decreasing function of $t$. By selecting a finite subcover, we readily obtain $T>0$ such that for all $t\geq T$,
\[
\varphi(t,K)\subset M_{\mu}\subset B_q,
\]
thereby establishing the proposition.
\end{proof}

To derive Theorem~\ref{T-1-7}, we invoke the following result affirming the Generalized Poincar$\acute{\text{e}}$ Conjecture for differentiable manifolds. It was proved by Smale~\cite{Smale-1961} in dimensions greater than $4$, by Freedman~\cite{Freedman-1982} in dimension $4$, and by Perelman~\cite{Perelman-2002,Perelman-2003.3,Perelman-2003.7} in dimension $3$. For dimensions lower than 3, it has long been known by the classification of manifolds in those dimensions.

\begin{theorem}[Smale, Freedman, Perelman]\label{T-2-10}
Every closed smooth $n$-manifold that has the homotopy type of $S^n$ is homeomorphic to $S^n$.
\end{theorem}

We are now in a position to show Theorem~\ref{T-1-7}.
\begin{proof}[\textbf{Proof of Theorem~\ref{T-1-7}}]
Assertion $(i)$ follows straightforwardly from Lemma~\ref{L-2-3} and Theorem~\ref{T-2-4}, while assertion $(ii)$ was established in Theorem~\ref{T-2-5}.

To address assertion $(iii)$, we first note that $M_c=\big\{x\in M: f(x)<c\big\}$ is a smooth manifold of dimension $d$. Define $
h=1/(c-f)$, which serves as a Lyapunov-Reeb function on $M_c$. By assertion $(ii)$ (equivalently, Theorem~\ref{T-2-5}), we infer that $M_c$ is diffeomorphic to $\mathbb{R}^d$.

%for any $p\in M\setminus\{q\}$, consider the maximal integral curve $\xi_p:I_p\to M\setminus\{q\}$
%satisfying the ordinary differential equation
%\[
%\frac{\mathrm d \xi_p}{ \mathrm d t} = -\frac{\nabla f}{\|\nabla f\|^2}
%\]
%with initial condition $\xi_p(0)=p$.
%the smooth map%$\xi^{-1}$

%Consequently,
%\[
%\mathcal U_c=\left(-\infty,c-f(q)\right)\times L_c\subset \mathcal U.
%\]

%Taking $p\in L_c$, we then obtain

%By Proposition~\ref{P-2-10} below,  $\psi$ is a homeomorphism.
%For any $p\in M\setminus\{q\}$, the Escape Lemma~\cite[Lemma 9.19]{Lee-2018} implies the maximal interval for $p$ is  $I_p=\big(-\infty, p-f(q)\big)$. Letting $p$ vary over $L_c$ defines a smooth map

We now characterize the topology of the level set
$L_c=\big\{x\in M: f(x)=c\big\}$.  Since $f$ is proper and $c>\inf f$ is a regular value, the Regular Value Theorem~\cite[Corollary 5.14]{Lee-2018} implies $L_c$ is a closed $(d-1)$-dimensional smooth manifold. In view of Theorem~\ref{T-2-10}, to prove $L_c$
is homeomorphic to $S^{d-1}$, we need to show $L_c$ has the homotopy type of $S^{d-1}$. To this end, we refine a technique by Wilson~\cite{Wilson-1967} to construct a homotopy equivalence between $L_c$ and $M\setminus\{q\}$, where $q$ denotes the unique critical point of $f$. Let $\psi\colon\mathcal U\to M\setminus\{q\}$ be the maximal flow generated by the vector field $-\nabla f/\|\nabla f\|^2$. A straightforward computation yields
\[
\frac{\mathrm d f\left(\psi(t,\cdot)\right)}{\mathrm d t}\equiv-1.
\]
By the Escape Lemma~\cite[Lemma 9.19]{Lee-2018}, the maximal existence interval of the integral curve starting at a point $p\in M\setminus\{q\}$ is given by $I_p=\big(-\infty, f(p)-f(q)\big)$. Varying $p$ over the level set $L_c$ yields a smooth map
\[
\begin{aligned}
\xi\colon\; \mathcal U_c\;\,&\longrightarrow\; M\setminus\{q\}\\
          (t,p)\;&\longmapsto\;\; \psi(t,p),
\end{aligned}
\]
where $\mathcal U_c=\big(-\infty,c-f(q)\big)\times L_c$. One readily checks that  $\xi\colon \mathcal U_c\to  M\setminus\{q\}$ is a diffeomorphism
with the inverse map
\[
\begin{aligned}
\xi^{-1}\colon\; M\setminus\{q\}\;&\longrightarrow\quad\mathcal U_c\\
          y\;\quad&\longmapsto\; \left(c-f(y),\psi\big(f(y)-c,y\big)\right).
\end{aligned}
\]
Since $M$ is diffeomorphic to $\mathbb{R}^d$,  both $M\setminus\{q\}$ and $\mathcal U_c$ have the homotopy type of $S^{d-1}$. Meanwhile, observing that $L_c$ is a deformation retract of $\mathcal U_c$, we prove $L_c$ is a closed smooth manifold homotopy equivalent to $S^{d-1}$. Overall, these relations are summarized as follows:
\[
L_c\simeq \mathcal U_c\cong_{\operatorname{diff}} M\setminus\{q\}\cong_{\operatorname{diff}}\mathbb R^d\setminus\{\mathbf{0}\}\simeq S^{d-1},
\]
where $\mathbf{0}\in \mathbb R^d$ denotes the origin, and the symbols ``$\simeq$'' and ``$\cong_{\operatorname{diff}}$'' denote homotopy equivalence and diffeomorphism between smooth manifolds, respectively. In light of Theorem~\ref{T-2-10}, we conclude that $L_c$ is homeomorphic to $S^{d-1}$.

It remains to consider the differentiable structure of $L_c$. For $d\leq 4$, the smooth structure of the $(d-1)$-sphere is unique (see~\cite{Moise-1952}), which implies $L_{c}\cong_{\operatorname{diff}} S^{d-1}$. For $d\geq 6$, using the $h$-Cobordism Theorem~\cite[Chap. 9, Proposition A]{Milnor-1965}, we also derive that $L_c$ is diffeomorphic to $S^{d-1}$. In summary, $L_c\cong_{\operatorname{diff}} S^{d-1}$ whenever $d\neq 5$. Now suppose the critical point $q$ of $f$ is nondegenerate. By the Morse Lemma~\cite{Milnor-1963}, we first establish that $L_{c_0}$ is diffeomorphic to $S^{d-1}$ for $c_0$ sufficiently close to $f(q)$. Since there exists a diffeomorphism between $L_c$ and $L_{c_0}$, it follows that $L_c\cong_{\operatorname{diff}} S^{d-1}$. This completes the proof.
\end{proof}

\begin{remark}
For $d=5$, the claim $L_c\cong_{\operatorname{diff}}S^4$ is implied by the smooth 4-dimensional Poincar$\acute{\text{e}}$ conjecture and remains unsolved. Whether these two statements are equivalent seems to be an interesting question, yet it is out of the reach of this paper.
\end{remark}

\section{Moduli spaces of arm linkages and configurations}\label{S-3}
In this section, we determine the topology of the moduli spaces of arm linkages and configurations.

As noted earlier, our strategy is to construct L-R functions on these spaces. To achieve this goal, we need to examine the smooth structures of $\mathcal M(R_m)$ and $\mathcal M(R_m,\ell)$. Let us begin with the moduli space $\mathcal M(R_m)$. For each $p\in \mathcal M(R_m)$, via an appropriate orientation-preserving isometry, we may assume
\[
\begin{aligned}
&p(v_1)=(0,0),\quad p(v_2)=(\rho_1,0),\\
&p(v_3)=p(v_2)+(\rho_2\cos\theta_2,\rho_2\sin\theta_2),\\
&\quad\qquad\qquad\qquad\cdots\\
&p(v_{m})=p(v_{m-1})+(\rho_{m-1}\cos\theta_{m-1},\rho_{m-1}\sin\theta_{m-1}),
\end{aligned}
\]
where $(\rho,\theta)
:=(\rho_1,\cdots,\rho_{m-1},\theta_2,\cdots,\theta_{m-1})\in \mathbb R_{+}^{m-1}\times\left(\mathbb R/2\pi\mathbb Z\right)^{m-2}$.
Noting that each
$p\in \mathcal M(R_m)$ is uniquely determined by $(\rho,\theta)$, we can identify $\mathcal M(R_m)$ with a subset of $\mathbb R_{+}^{m-1}\times\left(\mathbb R/2\pi\mathbb Z\right)^{m-2}$. Since the self-avoiding property is stable under small perturbations,  $\mathcal M(R_m)$ is actually an open subset of
$\mathbb R_{+}^{m-1}\times\left(\mathbb R/2\pi\mathbb Z\right)^{m-2}$, thus inheriting the induced smooth structure from the ambient space. Meanwhile,
\[
\mathcal M(R_m,\ell)=\left\{(\rho,\theta)\in\mathcal M(R_m):\rho=\vec{\ell}\right\},
\]
where $\vec{\ell}=\big(\ell([v_1,v_2]),\cdots,\ell([v_{m-1},v_m])\big)$ denotes the length vector. This implies that $\mathcal M(R_m,\ell)$ carries the smooth structure of an embedded submanifold of $\mathcal M(R_m)$.

In summary, we obtain the following lemma, part of which has appeared in the work of Shimamoto and Vanderwaart~\cite{Shimamoto-2005}.
% and Farber~\cite[Chap. 3]{Farber-2008}.
\begin{lemma}\label{L-3-1}
Let $m\geq 2$ and let $\ell:E\to\mathbb R_{+}$ be an arbitrary assignment. The following statements hold:
\begin{itemize}
\item[$(i)$]$\mathcal M(R_m)$ is a smooth manifold of dimension $2m-3$.
\item[$(ii)$]$\mathcal M(R_m,\ell)\subset \mathcal M(R_m)$ is a smooth submanifold of dimension $m-2$.
\end{itemize}
\end{lemma}

To establish that strain energies on $\mathcal M(R_m,\ell)$ are L-R functions, we invoke the following result by Connelly, Demaine and Rote~\cite[Theorem 3]{Connelly-2003}, which played a pivotal role in resolving the Carpenter's Rule Problem. For alternative proofs, see also the work of Rote, Santos and Streinu~\cite[Theorem 4.3]{Rote-2003}. We call an element in $\mathcal M(R_m,\ell)$ (resp. $\mathcal M(R_m)$) a \emph{straight linkage} (resp.~\emph{straight configuration}) if all its vertices are collinear.

\begin{theorem}[Connelly-Demaine-Rote]\label{T-3-2}
Suppose $p\in \mathcal M(R_m,\ell)$ is not the straight linkage. Then there exists an expansive motion $\{p_t\}_{-\varepsilon<t<\varepsilon}$ ($\varepsilon>0$) in $\mathcal M(R_m,\ell)$ such that $p_0=p$.
\end{theorem}

%Suppose $\Phi:\mathcal M(R_m,\ell)\to[0,+\infty)$ is a strain energy.

We proceed to show Theorem~\ref{T-1-11}.

\begin{proof}[\textbf{Proof of Theorem \ref{T-1-11}}]
It is easy to see that the strain energy $\Phi$ has a minimum point $q\in \mathcal M(R_m,\ell)$. We shall prove that $q$ is the unique critical point of $\Phi$, which confirms $\Phi$ is a Lyapunov-Reeb function.

First we claim $q$ is the straight linkage. If this is not the case, then Theorem~\ref{T-3-2} implies there exists an expansive motion $\{p_t\}_{-\varepsilon<t<\varepsilon}$ ($\varepsilon>0$) in
$\mathcal M(R_m,\ell)$ such that $p_0=q$. Since $\Phi$ is a strain energy, we have
\[
\frac{\mathrm{d}\Phi(p_t)}{\mathrm{d}t}\bigg|_{t=0}<0,
\]
contradicting the fact that $p_0=q$ is a minimum point of $\Phi$.

Next we verify that $\Phi$ has no critical points other than $q$. Assume, for contradiction, that $q'$ is another critical point of $\Phi$. Then $q'$ cannot be the straight linkage. Applying Theorem~\ref{T-3-2} again, we obtain an expansive motion $\{p'_t\}_{-\varepsilon'<t<\varepsilon'}$ ($\varepsilon'>0$) in $\mathcal M(R_m,\ell)$ with $p'_0=q'$. Because $q'$ is a critical point of $\Phi$, a straightforward computation yields
\[
\frac{\mathrm{d}\Phi(p'_t)}{\mathrm{d}t}\bigg|_{t=0}=0,
\]
which contradicts the definition of strain energies.

Meanwhile, the function $f=\Phi+\varpi$ is also a strain energy and hence a Lyapunov-Reeb function on $\mathcal M(R_m,\ell)$ with the straight linkage as its critical point. Owing to Proposition~\ref{P-3-3} below, we further check that the critical point is nondegenerate. Thus the theorem is proved.
\end{proof}

\begin{proposition}\label{P-3-3}
The Hessian of $f=\Phi+\varpi$ at the critical point is positive definite.
\end{proposition}
\begin{proof}
Recall that $f$ has the straight linkage $q$ as its unique critical point, at which $\Phi$ attains its minimum. This means the Hessian of $\Phi$ at $q$ is positive semi-definite, so it remains to verify that the Hessian of $\varpi$ at $q$ is positive definite. As mentioned before, $\mathcal M(R_m,\ell)$ is identified with an open subset $U_{m,\ell}\subset \left(\mathbb R/2\pi\mathbb Z\right)^{m-2}$ via the coordinates $\theta=(\theta_2,\cdots,\theta_{m-1})$. We thus view $\varpi$ as a function of $\theta$ and write
\[
\varpi=\frac{1}{\sum_{i=1}^{m-1}\ell_i^2+\sum_{i=2}^{m-1}2\ell_1\ell_i\cos\theta_i+\sum_{2\leq i<j\leq m-1}2\ell_i\ell_j\cos(\theta_i-\theta_j)},
\]
where $\ell_i=\ell\big([v_i,v_{i+1}]\big)$ for $i=1,\cdots,m-1$. Observe that the straight linkage $q$ corresponds to the null vector $\vec{0}:=(0,\cdots,0)\in \left(\mathbb R/2\pi\mathbb Z\right)^{m-2}$. Through a routine calculation, we find that the Hessian of $\varpi$ at $q$ takes the form
\[
2\varpi^2\left[\begin{array}{ccccccc}
\ell_2\left(\sum_{i=1}^{m-1}\ell_i\right) & \ell_2\ell_3  & \cdots & \ell_2\ell_{m-1}\\
\ell_3\ell_2 & \ell_3\left(\sum_{i=1}^{m-1}\ell_i\right)  & \cdots & \ell_3\ell_{m-1}\\
\vdots & \vdots  & \ddots &\vdots    \\
%0 & 0  & \cdots & 1  \\
\ell_{m-1}\ell_2 & \ell_{m-1}\ell_3
& \cdots & \ell_{m-1}\left(\sum_{i=1}^{m-1}\ell_i\right)
\end{array}\right].
\]
The matrix is symmetric, strictly diagonally dominant with positive diagonal entries, and therefore positive definite. The proposition thus follows.
\end{proof}

\begin{comment}
\begin{remark}
For any strain energy $\Phi$ on $\mathcal M(R_m,\ell)$, the function
\[
f(p)=-\|p(v_1)-p(v_m)\|+\Phi(p)
\]
defines a new strain energy whose critical point is non-degenerate.
\end{remark}
\end{comment}

To construct L-R functions on $\mathcal M(R_m)$, we devise a projection method. Precisely, for each $p\in \mathcal M(R_m)$, let $\varsigma(p)\in \mathcal M(R_m)$ be the straight configuration sharing the same length vector as $p$. This induces the following projection map
\[
\begin{aligned}
\varsigma\;:\;\mathcal M(R_m)\;&\longrightarrow\;\mathcal L(R_m)\\
p\;\;\;\,&\longmapsto\;\;\varsigma(p),
\end{aligned}
\]
where $\mathcal L(R_m)\subset\mathcal M(R_m)$ denotes the subspace of straight configurations. In terms of coordinates, $\mathcal L(R_m)$ is characterized by
\[
\mathcal L(R_m)=\left\{(\rho,\theta)\in\mathcal M(R_m):\theta=\vec{0}\right\}.
\]
Obviously, $\mathcal L(R_m)$ is a smooth manifold diffeomorphic to $\mathbb R_{+}^{m-1}$.

%Here $\vec{0}:=(0,\cdots,0)\in \left(\mathbb R/2\pi\mathbb Z\right)^{m-2}$ represents the null vector.
\begin{definition}
A smooth map $\Phi:\mathcal M(R_m)\to [0,+\infty)$ is called a
\textbf{$\varsigma$-potential} on $\mathcal M(R_m)$ if the following properties hold:
\begin{itemize}
\item[\textbf{(a1)}] For every $q\in \mathcal L(R_m)$, the restriction of $\Phi$ to $\varsigma^{-1}(q)$ is a strain energy;
\item[\textbf{(a2)}]For every compact set $K\subset \mathcal L(R_m)$, the restriction of $\Phi$ to $\varsigma^{-1}(K)$ is proper.
\end{itemize}
\end{definition}

\begin{remark}\label{R-3-5}
For $m\geq 3$, one readily checks that the function
\[
\Phi(p)=\sum_{\substack{[v_i,v_j]\in E \\ v_k\in V\setminus\{v_i,v_j\}}}\frac{1}
{\left(\|p(v_i)-p(v_k)\|+\|p(v_j)-p(v_k)\|-\|p(v_i)-p(v_j)\|\right)^2}
\]
mentioned in Remark~\ref{R-1-10} serves as a $\varsigma$-potential on $\mathcal M(R_m)$.
\end{remark}

With these concepts, we acquire L-R functions on
$\mathcal M(R_m)$ as follows.

\begin{theorem}\label{T-3-6}
Let $m\geq 3$. Assume $\Phi:\mathcal M(R_m)\to [0,+\infty)$ is a $\varsigma$-potential and assume $h:\mathcal L(R_m)\to\mathbb R$ is a Lyapunov-Reeb function. Then
\[
f=\Phi-\Phi\circ\varsigma+h\circ\varsigma
\]
is a Lyapunov-Reeb function on $\mathcal M(R_m)$.
\end{theorem}

\begin{proof}
For each $p\in \mathcal M(R_m)$, since $\varsigma(p)$ is the straight configuration sharing the same length vector as $p$, we derive from
property~\textbf{(a1)} and Theorem~\ref{T-1-11} that
\begin{equation}\label{E-3-1}
\Phi(p)\geq \Phi\big(\varsigma(p)\big).
\end{equation}
Given that $h$ is lower-bounded,  $f$ must also be lower-bounded.

Next we prove $f$ is proper. For $c_1=\inf f$ and $c_2>\inf f$, it suffices to show that any sequence $\{p_n\}$ contains a convergent subsequence, provided
\begin{equation}\label{E-3-2}
c_1\leq f(p_n)=\Phi(p_n)-\Phi\big(\varsigma(p_n)\big)
+h\big(\varsigma(p_n)\big)\leq c_2.
\end{equation}
Because $\Phi$ satisfies property~\textbf{(a2)} and $h$ is a Lyapunov-Reeb function on $\mathcal L(R_m)$, we reduce the proof to verifying that both $\big\{h(\varsigma(p_n))\big\}$ and $\big\{\Phi(p_n)\big\}$ are upper-bounded. From~\eqref{E-3-1} and~\eqref{E-3-2}, it is straightforward to see
\[
h\big(\varsigma(p_n)\big)\leq c_2.
\]
By the properness of $h$, we deduce that $\big\{\varsigma(p_n)\big\}$ lies in the compact set
\[
K:=h^{-1}\big([b_2,c_2]\big)\subset \mathcal L(R_m),
\]
where $b_2=\min_{q\in \mathcal L(R_m)} h(q)$. Reapplying~\eqref{E-3-2}, we obtain
\[
\Phi(p_n)\leq c_2-b_2+\max\nolimits_{q\in K}\Phi(q).
\]

It remains to demonstrate that  $f$ possesses exactly one critical point. To do this, we regard $f$ as a function of $(\rho,\theta)$ and infer that
\[
f(\rho,\theta)=\Phi(\rho,\theta)-\Phi(\rho,\vec{0})+h(\rho,\vec{0}).
\]
Suppose $(\tilde{\rho},\tilde{\theta})$ is a critical point of $f$. For $i=2,\cdots,m-1$, taking the derivative with respect to $\theta_i$ gives
\[
\frac{\partial f}{\partial \theta_i}(\tilde{\rho},\tilde{\theta})=\frac{\partial \Phi}{\partial \theta_i}(\tilde{\rho},\tilde{\theta})=0,
\]
which implies $\tilde{\theta}$ is a critical point of  $\Phi(\tilde{\rho},\cdot)$. By property~\textbf{(a1)} and Theorem~\ref{T-1-11}, this yields $\tilde{\theta}=\vec{0}$. For $i=1,\cdots,m-1$, the derivative with respect to $\rho_i$ satisfies
\[
\frac{\partial f}{\partial \rho_i}(\tilde{\rho},\tilde{\theta})=\frac{\partial \Phi}{\partial \rho_i}(\tilde{\rho},\tilde{\theta})-\frac{\partial \Phi}{\partial \rho_i}(\tilde{\rho},\vec{0})+\frac{\partial h}{\partial \rho_i}(\tilde{\rho},\vec{0})=\frac{\partial h}{\partial \rho_i}(\tilde{\rho},\vec{0})=0.
\]
It follows that $\tilde{\rho}=\rho^\ast$, where $\rho^\ast$ is the unique critical point of $h(\cdot,\vec{0})$. Thus, $f$ has exactly one critical point $(\tilde{\rho},\tilde{\theta})=(\rho^\ast,\vec{0})$. This completes the proof.
\end{proof}

\begin{remark}\label{R-3-7}
As an example, we observe that
\[
h(q)=\sum\nolimits_{i=1}^{m-1}\log^2\|q(v_{i+1})-q(v_i)\|
\]
is a Lyapunov-Reeb function on $\mathcal L(R_m)$.
\end{remark}

Let us consider the topological characterizations of the relevant spaces.

\begin{proof}[\textbf{Proof of Theorem~\ref{T-1-3}}]
We assume $m\geq 3$, as the case $m=2$ is trivial. By virtue of Theorems~\ref{T-1-11}, ~\ref{T-3-6} and Remarks~\ref{R-1-10},~\ref{R-3-5},~\ref{R-3-7}, both
$\mathcal M(R_m,\ell)$ and $\mathcal M(R_m)$ admit Lyapunov-Reeb functions. In view of Theorem~\ref{T-1-7}, the conclusions follow.
\end{proof}

\begin{proof}[\textbf{Proof of Corollary~\ref{C-1-13}}]
It is a consequence of property $(iii)$ of Theorem~\ref{T-1-7}.
\end{proof}

%\subsection{Straightening arm linkages and configurations}
We now develop algorithms for straightening arm linkages and configurations.
Recall that $\mathcal M(R_m,\ell)$ can be identified with an open subset $U_{m,\ell}\subset \left(\mathbb R/2\pi\mathbb Z\right)^{m-2}$ via the coordinates $\theta=(\theta_2,\cdots,\theta_{m-1})$. Let $m\geq 3$ and let $\Phi$ be given as in Remark~\ref{R-1-10}. Viewing $\Phi$ as a function of $\theta$, we infer from Theorem~\ref{T-1-11} that $\Phi$ is an L-R function on $U_{m,\ell}$ with $\vec{0}\in U_{m,\ell}$ being the critical point. Consider the negative gradient flow
\begin{equation}\label{E-3-3}
\frac{\mathrm d\theta_i}{\mathrm d t}=-\frac{\partial \Phi}{\partial \theta_i},\quad i=2,\cdots,m-1.
\end{equation}

From Theorem~\ref{T-1-7}, we derive the following result.

\begin{theorem}
Suppose $m\geq 3$. The flow~\eqref{E-3-3} starting at any point in $U_{m,\ell}$ exists for all $t\geq 0$ and converges to $\vec{0}$ as $t\to+\infty$.
\end{theorem}

Let $\theta(t)$ be the solution to~\eqref{E-3-3} and let $p_t\in \mathcal M(R_m,\ell)$ be given by
\[
\begin{aligned}
& p_t(v_1)=(0,0),\quad p_t(v_2)=(\ell_1,0),\\
& p_t(v_i)=
\left(\ell_1+\sum\nolimits_{j=2}^{i-1}\ell_j\cos\theta_j(t),
\sum\nolimits_{j=2}^{i-1}\ell_j\sin\theta_j(t)\right),\quad i=3,\cdots,m.
\end{aligned}
\]
%where $\ell_k=\ell\big([v_k,v_{k+1}]\big)$ for $k=1,\cdots,m-1$.
Then $\{p_t\}_{0\leq t\leq +\infty}$ yields the desired straightening motion in $\mathcal M(R_m,\ell)$ (see Fig.~\ref{F-2}).

\begin{figure}[htbp]
    \centering
    \includegraphics[width=0.72\textwidth]{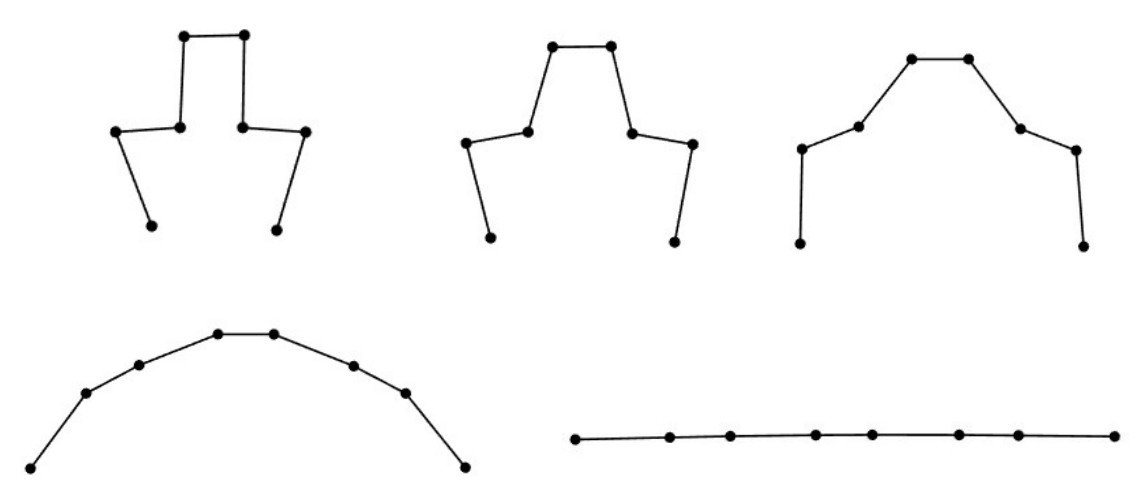}
    \caption{Straightening an arm linkage}
    \label{F-2}
\end{figure}

Analogously, we identify $\mathcal M(R_m)$ with an open subset $U_m\subset \mathbb R_{+}^{m-1}\times\left(\mathbb R/2\pi\mathbb Z\right)^{m-2}$ via the coordinates $(\rho,\theta)=(\rho_1,\cdots,\rho_{m-1},\theta_2,\cdots,\theta_{m-1})$. Let $\Phi$ and $h$ be given as in Remarks~\ref{R-1-10} and~\ref{R-3-7}, respectively. By virtue of Theorem~\ref{T-3-6},
\[
f(\rho,\theta)=\Phi(\rho,\theta)-\Phi(\rho,\vec{0})+h(\rho,\vec{0})
\]
is an L-R function on $U_m$. Let $\big(\rho(t),\theta(t)\big)$ be the negative gradient flow of $f$ starting at a given vector in $U_m$, and let $p_t\in\mathcal M(R_m)$ be the configuration parameterized by $\big(\rho(t),\theta(t)\big)$. Then $\{p_t\}_{0\leq t\leq +\infty}$ forms a straightening motion in $\mathcal M(R_m)$.

\section{Moduli spaces of cycle linkages and configurations}\label{S-4}
We further investigate the moduli spaces of cycle linkages and configurations. As before, we need to equip
$\mathcal M^{+}(C_m)$ and $\mathcal M^{+}(C_m,\ell)$ with smooth structures and construct L-R functions. To ensure
$\mathcal M^{+}(C_m,\ell)\neq \emptyset$, we assume throughout this section that $\ell:E\to\mathbb R_{+}$ satisfies condition~$\mathrm{\mathbf{(c1)}}$.

\subsection{Smooth structures of the moduli spaces}
We first note that $\mathcal M^{+}(C_m)$ is a non-empty open subset of
$\mathcal M(R_m)$, so it inherits the induced smooth structure from $\mathcal M(R_m)$.
Furthermore, we find that $\mathcal M^{+}(C_m,\ell)$ carries the smooth structure of an embedded submanifold of $\mathcal M^{+}(C_m)$.
%For deeper results on this topic, the reader is referred to the seminal works of Kapovich and Millson~\cite{Kapovich-1995, Kapovich-1996,Kapovich-2002}.

\begin{lemma}\label{L-4-1}
Suppose $m\geq 3$ and suppose $\ell:E\to\mathbb R_{+}$ satisfies
condition~$\mathrm{\mathbf{(c1)}}$. The following hold:
\begin{itemize}
\item[$(i)$]$\mathcal M^{+}(C_m)$ is a smooth manifold of dimension $2m-3$
\item[$(ii)$]$\mathcal M^{+}(C_m,\ell)\subset \mathcal M^{+}(C_m)$ is a smooth submanifold of dimension $m-3$.
\end{itemize}
\end{lemma}

\begin{proof}
Since $\mathcal M^{+}(C_m)$ is a non-empty open subset of $\mathcal M(R_m)$, we deduce part $(i)$ from Lemma~\ref{L-3-1}. For part $(ii)$, we define a map $F:\mathcal M^{+}(C_m)\to \mathbb R_{+}^{m}$ by
\[
F(p)=\big(\|p(v_2)-p(v_1)\|,\cdots,\|p(v_m)-p(v_{m-1})\|,\|p(v_1)-p(v_m)\|\big).
\]
Then $\mathcal M^{+}(C_m,\ell)=F^{-1}(\vec{\ell})$, where
\[
\vec{\ell}=\big(\ell([v_1,v_2]),\cdots,\ell([v_{m-1},v_m]),\ell([v_{m},v_1])\big).
\]
Under condition $\mathrm{\mathbf{(c1)}}$, $\mathcal M^{+}(C_m,\ell)\neq \emptyset$. The lemma follows as a consequence of the Regular Value Theorem~\cite{Lee-2018}, provided that $\vec{\ell}$  is a regular value of the map $F$. It therefore suffices to show the Jacobian $DF(p)$ has full rank $m$ for each $p\in F^{-1}(\vec{\ell})$. By invoking the coordinates $(\rho,\theta)$ for $p$, we infer that $DF(p)$ has the form
\[
\left[\begin{array}{ccccccc}
1 & 0  & \cdots & 0 & 0 & \cdots & 0\\
0 & 1  & \cdots & 0 & 0 & \cdots & 0\\
& \vdots  & \ddots &   & & \vdots & \\
0 & 0  & \cdots & 1 & 0 & \cdots & 0 \\
\frac{\partial\rho_m}{\partial\rho_1} & \frac{\partial\rho_m}{\partial\rho_2}
& \cdots & \frac{\partial\rho_m}{\partial\rho_{m-1}} & \frac{\partial\rho_m}{\partial\theta_2}
& \cdots & \frac{\partial\rho_m}{\partial\theta_{m-1}}
\end{array}\right],
\]
where
\[
\begin{aligned}
&\rho_m:=\|p(v_1)-p(v_m)\|
=\sqrt{\left(\rho_1+\sum\nolimits_{j=2}^{m-1}
\rho_j\cos\theta_j\right)^2
+\left(\sum\nolimits_{j=2}^{m-1}\rho_j\sin\theta_j\right)^2}.
\end{aligned}
\]
We claim $DF(p)$ has rank $m$. Assume this does not hold. For $i=2,\cdots,m-1$, then
\[
\frac{\partial\rho_m}{\partial\theta_i}=
\frac{-\left(\rho_1+\sum\nolimits_{j=2}^{m-1}
\rho_j\cos\theta_j\right)\rho_i\sin\theta_i
+\left(\sum\nolimits_{j=2}^{m-1}\rho_j\sin\theta_j\right)
\rho_i\cos\theta_i}{\rho_m}=0,
\]
which is equivalent to saying the vector
\[
p(v_m)-p(v_1)
=\left(\rho_1+\sum\nolimits_{j=2}^{m-1}\rho_j\cos\theta_j,
\sum\nolimits_{j=2}^{m-1}\rho_j\sin\theta_j\right)
\]
is parallel to
\[
p(v_{i+1})-p(v_i)=\left(\rho_i\cos\theta_i,\rho_i\sin\theta_i\right).
\]
As a result, $p(v_1),p(v_2),\cdots,p(v_m)$ are collinear, contradicting that they are the vertices of a simple polygon.

We thus verify that $DF(p)$ has rank $m$, which concludes
 $\mathcal M^{+}(C_m,\ell)=F^{-1}(\vec{\ell})$ is an embedded smooth submanifold of $\mathcal M^{+}(C_m)$ of dimension $m-3$, as desired.
\end{proof}

\subsection{Polygon triangulations and local coordinates}
Next we introduce a system of local coordinates for the moduli spaces via polygon triangulations. To achieve this, we invoke the following classical result. For a proof, we refer the reader to the celebrated Two Ears Theorem attributed to Meisters~\cite{Meisters-1975}. %See also the monograph by Devadoss and O'Rourke~\cite[Theorem 1.4]{Devadoss-2011} for an alternative approach.

\begin{lemma}\label{L-4-2}
Every simple polygon in $\mathbb R^2$ admits a triangulation without adding vertices. Moreover, if the polygon has $m$ vertices, then each of such triangulations has $m-2$ triangles and $m-3$ diagonals.
\end{lemma}

\begin{figure}[htbp]
     \centering
\begin{tikzpicture}[scale=0.5]
\coordinate (a) at (0.2,-1.3);
\coordinate (b) at (2.2,0);
\coordinate (c) at (2.4,2.2);
\coordinate (d) at (0.3,1.5);
\coordinate (e) at (-1.9,2.6);
\coordinate (f) at (-2.5,-0.2);
\coordinate (g) at (-1,0.18);
\coordinate (h) at (-1,1);
\coordinate (i) at (0,1.4);
\coordinate (k) at (0,3);
\draw (a)--(b)--(c)--(d)--(e)--(f)--(g)--(a);
\draw[dashed](d)--(b);
\draw[dashed](a)--(d);
\draw[dashed](d)--(g);
\draw[dashed](e)--(g);

\fill (a) circle (2pt);
\fill (b) circle (2pt);
\fill (c) circle (2pt);
\fill (d) circle (2pt);
\fill (e) circle (2pt);
\fill (f) circle (2pt);
\fill (g) circle (2pt);
\node[below] at (1.2,0.8) {$e$};
%\draw ([shift={(0.2,-1.1)}]0:0.4) arc[radius=0.5, start angle=45, end angle=90];
\draw ([shift={(0.21,-1)}]0:0.4)
arc[radius=0.5, start angle=32, end angle=80];
%\draw ([shift={(2.48,2.55)}]-125:0.6) arc[radius=0.3, start angle=-160, end angle= %-95];
\draw ([shift={(2.47,2.56)}]-125:0.6)
arc[radius=0.3, start angle=-166, end angle= -106];
\node[below] at (0.7,0) {$\gamma$};
\node[below] at (2,2) {$\beta$};

\end{tikzpicture}
\caption{Polygon triangulation with diagonals marked by dotted lines}
\label{F-3}
\end{figure}

We further present a refinement of Lemma~\ref{L-4-2},  serving as a tool for proving the area-increasing property in the next subsection (see Lemma~\ref{L-4-9}). Suppose
$\mathcal T$ is a triangulation of a simple polygon without adding vertices. A diagonal $e$ in $\mathcal T$ is called  a \emph{Delaunay-Lawson edge} if $\beta+\gamma\leq \pi$, where $\beta,\gamma$ are the angles opposite $e$ in the two adjacent triangles of $\mathcal T$ (see Fig.~\ref{F-3}). The lemma below follows from Lawson's Local Optimization Procedure~\cite{Lawson-1984}. %Improved results can also be found in the works of De Floriani, Falcidieno and Pienovi~\cite{Floriani-1985}, Lee and Lin~\cite{Lee-1986}, Chew~\cite{Chew-1989}, Chin and Wang~\cite{Chin-1999}, and others.

\begin{lemma}\label{L-4-3}
Every simple polygon in $\mathbb R^2$ admits a triangulation such that no extra vertices are added and all diagonals are Delaunay-Lawson edges.
\end{lemma}

Recall that each $p\in \mathcal M^{+}(C_m)$ gives a simple polygon $p(C_m)$. By Lemma~\ref{L-4-2}, $p(C_m)$ admits a triangulation $\mathcal T$ without adding vertices. Let $U_{\mathcal T}\subset \mathcal M^{+}(C_m)$ be a neighborhood of $p$ such that every $p'\in U_{\mathcal T}$ yields a simple polygon admitting a triangulation of the same combinatorial type as $\mathcal T$. We thus obtain a map
\[
\begin{aligned}
\varphi_{\mathcal T}:\;&U_{\mathcal T}\quad \longrightarrow\qquad\qquad\quad \mathbb R_{+}^{2m-3}\\
&\;p\;\;\quad \longmapsto\;\left(l_1(p),\cdots,l_{m}(p),\lambda_{1}(p),\cdots,\lambda_{m-3}(p)\right),
\end{aligned}
\]
where $l_1(p),\cdots,l_{m}(p)$ denote the side lengths of $p(C_m)$, and $\lambda_{1}(p),\cdots,\lambda_{m-3}(p)$ denote the lengths of the diagonals in the triangulation. Since
$\varphi_{\mathcal T}$ is a diffeomorphism from $U_{\mathcal T}$ to an open subset $\varphi_{\mathcal T}(U_{\mathcal T}) \subset \mathbb R_{+}^{2m-3}$,
$(U_{\mathcal T}, \varphi_{\mathcal T})$ forms a coordinate chart on $\mathcal M^{+}(C_m)$.
Moreover, for another chart $(U_{\mathcal T'}, \varphi_{\mathcal T'})$ with
$U_{\mathcal T}\cap U_{\mathcal T'}\neq \emptyset$,  the transition map
\[
\varphi_{\mathcal T'} \circ\varphi^{-1}_{\mathcal T}~:~\varphi_{\mathcal T}\left(U_{\mathcal T}\cap U_{\mathcal T'}\right)
~\longrightarrow~\varphi_{\mathcal T'}\left(U_{\mathcal T}\cap U_{\mathcal T'}\right)
\]
is smooth. Let $\Lambda$ be the set of combinatorial types of triangulations of all positively oriented simple $m$-gons. Then
$\mathcal A=\left\{(U_{\mathcal T}, \varphi_{\mathcal T})\right\}_{\mathcal T\in \Lambda}$ constitutes a smooth atlas for $\mathcal M^{+}(C_m)$,
with $(l,\lambda):=(l_1,\cdots,l_{m}, \lambda_1,\cdots,\lambda_{m-3})$ as local coordinates. Note that
\[
\mathcal M^{+}(C_m,\ell)=\left\{(l,\lambda)\in\mathcal M^{+}(C_m):l=\vec{\ell}\right\}.
\]
This implies that $\lambda_{1},\cdots,\lambda_{m-3}$ serve as local coordinates for points in $\mathcal M^{+}(C_m,\ell)$.

\subsection{Properties of the area function}
We now explore several properties of the area function.
An element in $\mathcal M^{+}(C_m,\ell)$ (resp. $\mathcal M^{+}(C_m)$) is called a \emph{cocircular linkage} (resp. \emph{cocircular configuration}) if it corresponds to a simple polygon inscribed in a circle. By Penner's Theorem~\cite[Theorem 6.2]{Penner-1987}, there exists exactly one cocircular linkage in
$\mathcal M^{+}(C_m,\ell)$. Relevant results also appeared in the works of Schlenker~\cite{Schlenker-2007}, Kou\v{r}imsk\'{a}, Skuppin and Springborn~\cite{Kour-Sku-Springborn-2016}, Xu and Zhou~\cite{Xu-2024}, and many others. The following is a classical result in geometry. For completeness, we provide a proof using the local coordinates.

%traces back to early investigations by Moula~\cite{Moula-1737}, %L'Huilier~\cite{Huilier-1782}, Steiner~\cite{Steiner-1841}, and  Blaschke~\cite{Blaschke-1956}. For completeness, we provide an independent proof using the local coordinates.

\begin{proposition}\label{P-4-4}
The area function $A:\mathcal M^{+}(C_m,\ell)\to\mathbb R_+$ has the cocircular linkage as its unique critical point.
\end{proposition}

We first establish some properties of the area function for triangles.
\begin{lemma}\label{L-4-5}
Suppose $\Delta$ is a triangle in $\mathbb R^2$ with side lengths $l_i,l_j,l_k$ and opposite angles $\alpha_i,\alpha_j,\alpha_k$, respectively. Let $A(\Delta)$ denote the area of $\Delta$. The following formulas hold:
\begin{itemize}
\item[$(i)$] $\displaystyle{\frac{\partial A(\Delta)}{\partial l_i}=\frac{l_i\cot\alpha_i}{2}}$;
\item[$(ii)$] $\displaystyle{\frac{\partial^2 A(\Delta)}{\partial l_i^2}=\frac{\cot\alpha_i}{2}
    -\frac{1}{2l_il_jl_k}\left(\frac{l_i}{\sin\alpha_i}\right)^3}$;
\item[$(iii)$] $\displaystyle{\frac{\partial^2 A(\Delta)}{\partial l_i \partial l_j}=\frac{1}{2l_il_jl_k}\left(\frac{l_i}{\sin\alpha_i}\right)^3\cos\alpha_k}$.
\end{itemize}
\end{lemma}

\begin{remark}\label{R-4-6}
Set $\displaystyle{\Upsilon_{ijk}=\frac{1}{2l_il_jl_k}
\left(\frac{l_i}{\sin\alpha_i}\right)^3}$.
By the Law of Sines,
\[
\Upsilon_{ijk}=\Upsilon_{jki}=\Upsilon_{kij}:=\Upsilon(\Delta).
\]
We present the following equivalent versions of the latter two formulas:
\begin{itemize}
\item[$(ii')$] $\displaystyle{\frac{\partial^2 A(\Delta)}{\partial l_i^2}=\frac{\cot\alpha_i}{2}-\Upsilon(\Delta)}$;
\item[$(iii')$] $\displaystyle{\frac{\partial^2 A(\Delta)}{\partial l_i \partial l_j}=\Upsilon(\Delta)\cos\alpha_k}$.
\end{itemize}
\end{remark}

The proof is based on direct computations, with details deferred to~\S\ref{S-7}. Using this lemma, we are ready to show Proposition~\ref{P-4-4}.

\begin{proof}[\textbf{Proof of Proposition~\ref{P-4-4}}]
We assume $m\geq 4$, as the case $m=3$ is trivial. For each $p\in \mathcal M^{+}(C_m,\ell)$, let $\mathcal T$ be a triangulation of the polygon $p(C_m)$ without adding vertices. Suppose $\lambda_1,\cdots,\lambda_{m-3}$ are the lengths of the diagonals in the triangulation.  For $i=1,\cdots,m-3$, it follows from Lemma~\ref{L-4-5} that
\begin{equation}\label{E-4-1}
\frac{\partial A}{\partial \lambda_i}=\frac{\lambda_i}{2}\left(\cot\beta_i+\cot\gamma_i\right)
=\frac{\lambda_i\sin(\beta_i+\gamma_i)}{2\sin \beta_i\sin\gamma_i},
\end{equation}
where $\beta_i, \gamma_i$ are the angles opposite the  $i$-th diagonal in the two adjacent triangles of $\mathcal T$. At the cocircular linkage, we have
\begin{equation}\label{E-4-2}
\beta_i+\gamma_i=\pi,
\end{equation}
which yields
\begin{equation}\label{E-4-3}
\frac{\partial A}{\partial \lambda_i}=0.
\end{equation}
Hence the cocircular linkage is a critical point of the area function.

Conversely, any critical point of $A$ must satisfy~\eqref{E-4-3}, which is equivalent to~\eqref{E-4-2} and  is fulfilled exclusively by the cocircular linkage. Since the cocircular linkage in $\mathcal M^{+}(C_m,\ell)$ is unique, the proposition follows.
\end{proof}

\begin{remark}%\label{R-4-7}
The area function actually attains its maximum at the cocircular linkage. This property can be derived using Steiner's idea of ``mirroring the dent''~\cite{Steiner-1838} combined with a compactness argument.
\end{remark}

The following proposition asserts the nondegeneracy of the critical point, which in turn enables us to establish the smoothness of the projection map in the next subsection (see Proposition~\ref{P-4-14}).

\begin{proposition}\label{P-4-8}
Suppose $m\geq 4$. The Hessian of $A:\mathcal M^{+}(C_m,\ell)\to \mathbb R_+$ at the critical point is negative definite.
\end{proposition}

\begin{proof}
By virtue of Proposition~\ref{P-4-4}, the only critical point of $A$ is the cocircular  linkage $q$, which yields a simple polygon $q(C_m)$ inscribed in a circle. This implies $q(C_m)$ is a strictly convex polygon, admitting a triangulation $\mathcal T$ with diagonals connecting vertices $v_{m-1}$ and $v_i$ for $i=1,\cdots,m-3$. As shown in Fig.~\ref{F-4}, we introduce the following notations:
\begin{itemize}
\item{} For $i=1,\cdots,m-3$, let $\lambda_i$ denote the length of the diagonal in $\mathcal T$ between vertices $v_{m-1}$ and $v_i$;
\item{} For $j=1,\cdots,m-2$, let $\Delta_j$ be the triangle in $\mathcal T$ with vertices $v_{j-1},v_{j},v_{m-1}$ (where $v_0:=v_m$), and let $\sigma_j,\delta_j,\phi_j$ represent the angles of  $\Delta_j$ at vertices $v_{j-1},v_j,v_{m-1}$, respectively.
\end{itemize}

\begin{figure}[htbp]
     \centering
\begin{tikzpicture}[scale=0.72]
\coordinate (b) at (1.6,-2.5);
\coordinate (c) at (3.36,-0);
\coordinate (d) at (2.9,2.3);
\coordinate (e) at (-0.2,3.9);
\coordinate (f) at (-2.81,2.41);
\coordinate (g) at (-3.35,-0.1);
\coordinate (a) at (-1.4,-2.6);

\draw (a)--(b)--(c)--(d)--(e)--(f)--(g)--(a)--(b);

\draw[dashed](f)--(b);
\draw[dashed](f)--(a);
\draw[dashed](f)--(d);
\draw[dashed](f)--(g);
\draw[dashed](f)--(c);

\fill (a) circle (2pt);
\fill (b) circle (2pt);
\fill (c) circle (2pt);
\fill (d) circle (2pt);
\fill (e) circle (2pt);
\fill (f) circle (2pt);
\fill (g) circle (2pt);
\node[below] at (-1.4,-2.7) {$v_1$};
\draw ([shift={(-1.4,-2.6)}]125:0.6) arc[radius=0.4, start angle=127, end angle= 98];
\draw ([shift={(-1.4,-2.6)}]104:0.45) arc[radius=0.4, start angle=100, end angle= 0];
\node[below] at (1.75,-2.6) {$v_2$};
\draw ([shift={(1.6,-2.5)}]180:0.4) arc[radius=0.4, start angle=170, end angle= 125];
\draw ([shift={(1.6,-2.5)}]130:0.4) arc[radius=0.4, start angle=130, end angle= 60];
\node[right] at (3.39,-0.1) {$v_3$};
\draw ([shift={(3.36,-0)}]235:0.4) arc[radius=0.4, start angle=230, end angle= 160];
\draw ([shift={(3.36,-0)}]160:0.5) arc[radius=0.4, start angle=170, end angle= 95];
\node[right] at (2.96,2.35) {$v_4$};
\draw ([shift={(2.9,2.3)}]-80:0.3) arc[radius=0.3, start angle=-80, end angle= -180];
\draw ([shift={(2.9,2.3)}]180:0.5) arc[radius=0.4, start angle=180, end angle= 150];
\node[above] at (-0.2,3.96) {$v_5$};
\draw ([shift={(-0.2,3.9)}]-30:0.3) arc[radius=0.3, start angle=-30, end angle= -150];
\node[left] at (-2.82,2.48) {$v_6$};
\draw ([shift={(-2.81,2.41)}]30:0.5) arc[radius=0.4, start angle=30, end angle= -5];
\draw ([shift={(-2.81,2.41)}]-2:0.7) arc[radius=0.4, start angle=-5, end angle=-35];
\draw ([shift={(-2.81,2.41)}]-23:0.5) arc[radius=0.4, start angle=-14, end angle=-44];
\draw ([shift={(-2.81,2.41)}]-50:0.65) arc[radius=0.6, start angle=-45, end angle=-70];
\draw ([shift={(-2.81,2.41)}]-75:0.5) arc[radius=0.6, start angle=-85, end angle=-105];
\node[left] at (-3.39,-0.1) {$v_7$};
\draw ([shift={(-3.35,-0.1)}]75:0.3) arc[radius=0.3, start angle=75, end angle=-52];
\node at (-2.74,1.25) {$\phi_1$};
\node at (-2.08,1.2) {$\phi_2$};
\node[below] at (-1.7,2.05) {$\phi_3$};
\node at (-1.3,2.12) {$\phi_4$};
\node at (-1.75,2.65) {$\phi_5$};

\node at (-2.4,-0.55) {$\Delta_1$};
\node at (-0.4,-1.4) {$\Delta_2$};
\node at (1.35,-0.45) {$\Delta_3$};
\node at (1.6,1.4) {$\Delta_4$};
\node at (0.5,3) {$\Delta_5$};

\node at (-1.88,-1.56) {$\delta_1$};
\node at (1,-2.2) {$\delta_2$};
\node at (2.7,-0.1) {$\delta_3$};
\node at (2.55,1.9) {$\delta_4$};
\node at (-0.2,3.35) {$\delta_5$};
\node at (-2.7,0.1) {$\sigma_1$};
\node at (-1.1,-1.95) {$\sigma_2$};
\node at (1.6,-1.9) {$\sigma_3$};
\node at (2.75,0.6) {$\sigma_4$};
\node at (1.8,2.55) {$\sigma_5$};
\node at (-1.75,-0.25) {$\lambda_1$};
\node at (-0.2,0) {$\lambda_2$};
\node at (0.45,1.45) {$\lambda_3$};
\node at (-0.3,2.6) {$\lambda_4$};

\draw[dashed] (0,0.5) circle (3.4);

\end{tikzpicture}
\caption{A cocircular polygon with a triangulation}
\label{F-4}
\end{figure}

Recall that $\lambda_1,\cdots,\lambda_{m-3}$ serve as local coordinates. From Lemma~\ref{L-4-5} and Remark~\ref{R-4-6}, it follows that
\[
\frac{\partial^2 A}{\partial \lambda_i^2}=\frac{\partial^2 A(\Delta_i)}{\partial \lambda_i^2}+\frac{\partial^2 A(\Delta_{i+1})}{\partial \lambda_i^2}=\frac{\sin(\sigma_i+\delta_{i+1})}{2\sin\sigma_i\sin\delta_{i+1}}-\Upsilon(\Delta_i)-\Upsilon(\Delta_{i+1}).
\]
At the cocircular linkage $q$, we have $\sigma_i+\delta_{i+1}=\pi$, which yields
\[
\frac{\partial^2 A}{\partial \lambda_i^2}\big(q\big)=-\Upsilon(\Delta_i)-\Upsilon(\Delta_{i+1})<0.
\]
For $1\leq j<k\leq m-3$, Lemma~\ref{L-4-5} and Remark~\ref{R-4-6} also imply that
\[
\frac{\partial^2 A}{\partial \lambda_j\partial \lambda_k}\big(q\big)
=\begin{cases}
\begin{aligned}
&\Upsilon(\Delta_{j+1})\cos\phi_{j+1},&k-j=1,\\
&0,  &k-j>1.
\end{aligned}
\end{cases}
\]
Therefore,
\[
\left|\frac{\partial^2 A}{\partial \lambda_i^2}\big(q\big)\right|-\sum_{j\neq i}\left|\frac{\partial^2 A}{\partial \lambda_i\partial \lambda_j}\big(q\big)\right|>0.
\]
In summary, the Hessian of $A$ at $q$ is a strictly diagonally dominate matrix with negative diagonal entries, and hence negative definite.
\end{proof}

The following area-increasing property was established by Connelly, Demaine and Rote~\cite[Theorem 7]{Connelly-2003}. Here we provide a simpler proof based on Lemma~\ref{L-4-3}.

\begin{lemma}[Connelly-Demaine-Rote]\label{L-4-9}
Assume $\{p_t\}_{-\varepsilon<t<\varepsilon}$ $(\varepsilon>0)$ is an expansive motion in $\mathcal M^{+}(C_m,\ell)$. Then
\[
\frac{\mathrm{d}A(p_t)}{\mathrm{d}t}\geq0,\;\forall\, t\in (-\varepsilon,\varepsilon).
\]
\end{lemma}
\begin{proof}
By Lemma~\ref{L-4-3}, we obtain a triangulation $\mathcal T$ of $p_t(C_m)$ with no extra vertices added and all diagonals being Delaunay-Lawson edges. For $i=1,\cdots,m-3$, let $\lambda_{i}$ be the length of the $i$-th diagonal in $\mathcal T$, and let $\beta_i, \gamma_i$ be the angles opposite the $i$-th diagonal in the two adjacent triangles of
$\mathcal T$. Since all diagonals in $\mathcal T$ are Delaunay-Lawson edges, it follows that
$\beta_i+\gamma_i\leq \pi$. Combining this with~\eqref{E-4-1}, we derive
\[
\frac{\partial A}{\partial \lambda_i}=\frac{\lambda_i\sin(\beta_i+\gamma_i)}{2\sin \beta_i\sin \gamma_i}\geq 0.
\]
Recall that $\{p_t\}_{-\varepsilon<t<\varepsilon}$ is an expansive motion. Then
\[
\frac{\mathrm{d}\lambda_i}{\mathrm{d}t}\geq 0.
\]
A direct computation yields
\[
\frac{\mathrm{d}A(p_t)}{\mathrm{d}t}=\sum\nolimits_{i=1}^{m-3}\frac{\partial A}{\partial \lambda_i}\frac{\mathrm{d}\lambda_i}{\mathrm{d}t}\geq 0.
\]
\end{proof}

\subsection{Constructing Lyapunov-Reeb functions} Remember that each $p\in \mathcal M^{+}(C_m)$ corresponds to a simple polygon $p(C_m)$ in $\mathbb R^2$, with $\alpha_i(p)$ denoting the interior angle at vertex $p(v_i)$. As mentioned in~\S~\ref{S-1},  constructing L-R functions on $\mathcal M^{+}(C_m,\ell)$ requires analyzing the function
\[
w=\frac{1}{m}\left(\sum\nolimits_{i=1}^m w_i\right),
\]
where
\[
w_i(p)=
\begin{cases}
\begin{aligned}
&\exp\left(\frac{1}{\pi-\alpha_i(p)}\right), &&\alpha_i(p)>\pi,\\
&0, &&\alpha_i(p)\leq\pi.
\end{aligned}
\end{cases}
\]
Note that $w(p)=0$ if $p\in \mathcal M^{+}(C_m,\ell)$ is a \emph{convex cycle linkage} (i.e., $p$ corresponds to a convex polygon). The following lemma is due to Shimamoto and Wootters~\cite{Shimamoto-2014}. For the sake of self-containment, we include a proof.

\begin{lemma}[Shimamoto-Wootters]\label{L-4-10}
Assume $\{p_t\}_{-\varepsilon<t<\varepsilon}$ $(\varepsilon>0)$ is an expansive motion in $\mathcal M^{+}(C_m,\ell)$. Then
\[
\frac{\mathrm{d}w(p_t)}{\mathrm{d}t}\bigg|_{t=0}\leq0.
\]
\end{lemma}

\begin{proof}
To establish the lemma, we proceed by showing that
\[
\frac{\mathrm{d}w_i(p_t)}{\mathrm{d}t}\bigg|_{t=0}\leq0
\]
for $i=1,\cdots,m$. If $\alpha_i(p_0)\leq \pi$, it is straightforward to see
\[
\frac{\mathrm{d}w_i(p_t)}{\mathrm{d}t}\bigg|_{t=0}=0.
\]
For $\alpha_i(p_0)>\pi$, the Law of Cosines implies
\[
\frac{\mathrm{d}\alpha_i(p_t)}{\mathrm{d}t}
=\frac{\|p_t(v_{i-1})-p_t(v_{i+1})\|}{l_il_{i+1}\sin \alpha_i(p_t)}\cdot\frac{\mathrm{d}}{\mathrm{d}t}~\|p_t(v_{i-1})-p_t(v_{i+1})\|,
\]
where $l_i=\ell\big([v_{i-1},v_{i}]\big)$, $l_{i+1}=\ell\big([v_{i},v_{i+1}]\big)$. Since $\{p_t\}_{-\varepsilon<t<\varepsilon}$ $(\varepsilon>0)$ is an expansive motion, the distance $\|p_t(v_{i-1})-p_t(v_{i+1})\|$ is non-decreasing, which yields
\[
\frac{\mathrm{d}\alpha_i(p_t)}{\mathrm{d}t}\bigg|_{t=0}\leq 0.
\]
As a result,
\[
\frac{\mathrm{d}w_i(p_t)}{\mathrm{d}t}\bigg|_{t=0}
=\exp\left(\frac{1}{\pi-\alpha_i(p_0)}\right)
\left(\frac{1}{\pi-\alpha_i(p_0)}\right)^2
\frac{\mathrm{d}\alpha_i(p_t)}{\mathrm{d}t}\bigg|_{t=0}\leq 0.
\]
This completes the proof.
\end{proof}

Furthermore, we invoke the following fundamental result due to Connelly, Demaine and Rote~\cite[Theorem 3]{Connelly-2003}. It is worth mentioning that an independent proof can be found in the work of Rote, Santos and Streinu~\cite[Theorem 4.3]{Rote-2003}.
\begin{theorem}[Connelly-Demaine-Rote]\label{T-4-11}
If $p\in \mathcal M^{+}(C_m,\ell)$ is not a convex cycle linkage, then there exists an expansive motion $\{p_t\}_{-\varepsilon<t<\varepsilon}$ ($\varepsilon>0$) in $\mathcal M^{+}(C_m,\ell)$ such that $p_0=p$.
\end{theorem}

Let us move on to showing Theorem~\ref{T-1-12}.

\begin{proof}[\textbf{Proof of Theorem~\ref{T-1-12}}]
First it is easy to see $f$ is lower-bounded. Next we prove the properness of $f$. For $c_1=\inf f\geq 0$ and $c_2>\inf f$, we need to verify that any sequence
\[
\{p_n\}\subset f^{-1}\big([c_1,c_2]\big)\subset \mathcal M^{+}(C_m,\ell)
\]
contains a convergent subsequence.
Because the map $\Phi:\mathcal M^{+}(C_m,\ell)\to [0,+\infty)$ is proper, it is equivalent to validating that $\big\{\Phi(p_n)\big\}$ has an upper bound.
Assume on the contrary that there exists a subsequence $\{p_{n_k}\}\subset \{p_n\}$ such that
\[
\Phi(p_{n_k})\to+\infty.
\]
This implies $\{p_{n_k}\}$ tends to a linkage with self-intersections. On the other hand, given that $\{p_{n_k}\}\subset f^{-1}([c_1,c_2])$, we have
\[
0\leq c_1\leq f(p_{n_k})=\frac{1}{A(p_{n_k})}+w(p_{n_k})\Phi(p_{n_k})\leq c_2.
\]
Consequently,
\[
A(p_{n_k})\geq \frac{1}{c_2}>0\quad\text{and}\quad w(p_{n_k})\to 0.
\]
Since $w(p)=0$ if and only if $\alpha_i(p)\leq \pi$ for $i=1,\cdots,m$, we derive that $\{p_{n_k}\}$ approaches a convex cycle linkage with positive area, which has no self-intersections. This leads to a contradiction, so $f$ must be proper.

To demonstrate that $f$ has a unique critical point, we decompose
$\mathcal M^{+}(C_m,\ell)$ into the union of two disjoint subsets:
$\mathcal M_1$ and $\mathcal M_2$, where
$\mathcal M_1$ consists of the convex cycle linkages in $\mathcal M^{+}(C_m,\ell)$, and $\mathcal M_2:=\mathcal M^{+}(C_m,\ell)\setminus\mathcal M_1$.

On $\mathcal M_1$, $w\equiv 0$, so the function $f$ reduces to $\displaystyle{1/A}$. By Proposition~\ref{P-4-4}, $f$ possesses a unique critical point on $\mathcal M_1$.

We claim, meanwhile, that $f$ has no critical points on $\mathcal M_2$. Assume, for contradiction, that $q\in \mathcal M_2$ is a critical point of $f$. Since $q$ is  non-convex, Theorem~\ref{T-4-11} guarantees the existence of an expansive motion $\{p_t\}_{-\varepsilon<t<\varepsilon}$ ($\varepsilon>0$) in $\mathcal M^{+}(C_m,\ell)$ with $p_0=q$. The derivative of $f$ along this motion is expressed as
\[
\frac{\mathrm{d}f(p_t)}{\mathrm{d}t}=
-\frac{1}{A^2}\frac{\mathrm{d}A(p_t)}{\mathrm{d}t}
+w(p_t)\frac{\mathrm{d}\Phi(p_t)}{\mathrm{d}t}+\Phi(p_t)\frac{\mathrm{d}w(p_t)}{\mathrm{d}t}.
\]
By Lemmas~\ref{L-4-9},~\ref{L-4-10} and the definition of strain energies, we obtain
\[
\frac{\mathrm{d}f(p_t)}{\mathrm{d}t}\bigg|_{t=0}<0,
\]
which contradicts the assumption that $p_0=q$ is a critical point of $f$.

Therefore, $f$ has exactly one critical point on
$\mathcal M^{+}(C_m,\ell)$. By Proposition~\ref{P-4-8}, this critical point is nondegenerate. We thus finish the proof.%Putting the arguments together, we prove the theorem.
\end{proof}

\begin{remark}\label{R-4-12}
From the above proof, we can see that the function $f=1/A+w\Phi$ attains its global minimum at the cocicular linkage.
\end{remark}

\begin{remark}\label{R-4-13}
Shimamoto and Wootters~\cite{Shimamoto-2014} introduced the function $F=w\Phi$ to study the moduli space. Their proof relied crucially on a compactness argument regarding the set of convex cycle linkages. However, compactness may fail if the constraint\textemdash that for every sign assignment $\sigma: E\to\{-1,1\}$ to the edges, $\sum_{e_i\in E}\sigma(e_i)\ell(e_i)\neq0$\textemdash is removed. Fortunately, by virtue of Theorem~\ref{T-1-7}, we need not resort to specific topological analysis of the set of convex cycle linkages.
\end{remark}

To proceed, we further construct L-R functions on the moduli space
$\mathcal M^{+}(C_m)$. For each $p\in \mathcal M^{+}(C_m)$, Penner's Theorem~\cite[Theorem 6.2]{Penner-1987} implies there exists a unique cocircular configuration
$\tau(p)\in \mathcal M^{+}(C_m)$ sharing the same length vector as $p$. This gives rise to the following projection map
\[
\begin{aligned}
\tau\;:\;\mathcal M^{+}(C_m)\;&\longrightarrow\;\mathcal Q^{+}(C_m)\\
p\;\;\;\,&\longmapsto\;\;\tau(p),
\end{aligned}
\]
where $\mathcal Q^{+}(C_m)\subset\mathcal M^{+}(C_m)$ denotes the subspace of cocircular configurations. It is necessary to analyze the smoothness of this map.

\begin{proposition}\label{P-4-14}
$\mathcal Q^{+}(C_m)\subset\mathcal M^{+}(C_m)$ is a smooth submanifold of dimension $m$. Moreover, the map $\tau:\mathcal M^{+}(C_m)\to\mathcal Q^{+}(C_m)$ is smooth.
\end{proposition}

\begin{proof}
Without loss of generality, we assume $m\geq 4$. Given $p\in\mathcal M^{+}(C_m)$, let $\mathcal T$ be a triangulation of the polygon $p(C_m)$ without adding vertices, and let $(l,\lambda)$ be the local coordinates associated to
$\mathcal T$. Note that $\tau(p)$ corresponds to a cocircular polygon admitting a triangulation combinatorially equivalent to
 $\mathcal T$. According to Proposition~\ref{P-4-4}, $\tau(p)$ is parameterized by $\big(l,\lambda^\ast(l)\big)$, where $\big(l,\lambda^\ast(l)\big)$ satisfies
\[
\frac{\partial A}{\partial \lambda_i}\big(l,\lambda^\ast(l)\big)=0\quad\text{for}\quad i=1,\cdots,m-3.
\]
By invoking Proposition~\ref{P-4-8} and the Implicit Function Theorem, we deduce that $\lambda^\ast(l)$ depends smoothly on $l$, which completes the proof.
\end{proof}

%As in the preceding section, we introduce the following %concept.

\begin{definition}
A smooth map $\Phi:\mathcal M^{+}(C_m)\to [0,+\infty)$ is called a \textbf{$\tau$-potential}
on $\mathcal M^{+}(C_m)$ if the following properties hold:
\begin{itemize}
\item[\textbf{(s1)}] For every $q\in \mathcal Q^{+}(C_m)$, the restriction of $\Phi$ to $\tau^{-1}(q)$ is a strain energy;
\item[\textbf{(s2)}]For every compact set $K\subset \mathcal Q^{+}(C_m)$, the restriction of $\Phi$ to $\tau^{-1}(K)$ is proper.
\end{itemize}
\end{definition}

We construct L-R functions on $\mathcal M^{+}(C_m)$ as follows.
\begin{theorem}\label{T-4-16}
Let $m\geq 3$. Assume $\Phi:\mathcal M^{+}(C_m)\to [0,+\infty)$ is a $\tau$-potential and assume $h:\mathcal Q^{+}(C_m)\to\mathbb R$ is a Lyapunov-Reeb function. Then
\[
f=\frac{1}{A}-\frac{1}{A\circ\tau}+w\Phi+h\circ\tau
\]
is a Lyapunov-Reeb function on $\mathcal M^{+}(C_m)$.
\end{theorem}
\begin{proof}
For every $p\in \mathcal M^{+}(C_m)$, $\tau(p)$ is automatically convex, and thus $w\big(\tau(p)\big)\equiv0$. In light of property~\textbf{(s1)}, Theorem~\ref{T-1-12} and Remark~\ref{R-4-12}, we have
\begin{equation}\label{E-4-4}
\frac{1}{A(p)}+w(p)\Phi(p)\geq \frac{1}{A\big(\tau(p)\big)}
+w\big(\tau(p)\big)\Phi\big(\tau(p)\big)
=\frac{1}{A\big(\tau(p)\big)}.
\end{equation}
Since $h$ is lower-bounded, $f$ must also be lower-bounded.

To verify that $f$ possesses exactly one critical point, we decompose
\[
\mathcal M^{+}(C_m)=\mathcal Z_1\cup \mathcal Z_2,
\]
where $\mathcal Z_1\subset \mathcal M^{+}(C_m)$ denotes the subset of convex cycle configurations, and
$\mathcal Z_2$ is defined as the complement $\mathcal M^{+}(C_m)\setminus\mathcal Z_1$.
On $\mathcal Z_1$, observe that
\[
f=\frac{1}{A}-\frac{1}{A\circ\tau}+h\circ\tau.
\]
Following the proof of Theorem~\ref{T-3-6},  we assert $f$ has a unique critical point on $\mathcal Z_1$. Meanwhile, analogous arguments to those in Theorem~\ref{T-1-12} imply that $f$ has no critical points on $\mathcal Z_2$. Therefore, $f$ has exactly one critical point on $\mathcal M^{+}(C_m)$.

It remains to show that $f$ is proper. For $c_1=\inf f$ and $c_2>c_1$, we need to demonstrate that any sequence $\{p_n\}$ satisfying
\begin{equation}\label{E-4-5}
c_1\leq f(p_{n})=\frac{1}{A(p_{n})}
-\frac{1}{A\big(\tau(p_n)\big)}+w(p_{n})\Phi(p_{n})
+h\big(\tau(p_n)\big)
\leq c_2
\end{equation}
contains a convergent subsequence. Combining~\eqref{E-4-4} and~\eqref{E-4-5} yields
\[
h\big(\tau(p_n)\big)\leq c_2.
\]
Given that $h$ is a Lyapunov-Reeb function, $\big\{\tau(p_n)\big\}$ lies within a compact subset of $\mathcal Q^{+}(C_m)$, so $\big\{A\big(\tau(p_n)\big)\big\}$ is bounded both below and above by positive constants. This implies that there exists $c_3>0$ such that
\begin{equation}\label{E-4-6}
0<\frac{1}{A(p_{n})}+w(p_{n})\Phi(p_{n})\leq c_3.
\end{equation}
In view of property~\textbf{(s2)}, to confirm $\{p_n\}$ contains a convergent subsequence, it suffices to check that $\big\{\Phi(p_{n})\big\}$ is upper-bounded. Due to~\eqref{E-4-6}, this follows by analogy with the proof of Theorem~\ref{T-1-12}. Hence $f$ is proper. In conclusion, we establish that $f$ is a Lyapunov-Reeb function on $\mathcal M^{+}(C_m)$.
\end{proof}

Note that $\mathcal Q^{+}(C_m)$ is diffeomorphic to a non-empty open convex set in $\mathbb R_{+}^m$,  where each vector $(l_1,\cdots,l_m)$ satisfies
$l_i<\sum\nolimits_{j\neq i} l_j$ for $i=1,\cdots,m$. Thus, theoretically,
$\mathcal Q^{+}(C_m)$ admits L-R functions. We now construct a specific L-R function that will be applied to developing algorithms for the refolding problem in the next section. For $i=1,\cdots,m$, the length coordinate $l_i$ serves as a smooth function on $\mathcal Q^{+}(C_m)$. We relate these to L-R functions on
$\mathcal Q^{+}(C_m)$ as follows.

\begin{lemma}\label{L-4-17}
Let $L=\sum_{i=1}^m l_i$. Then
\[
h=\log^2 L-\left(\sum\nolimits_{i=1}^{m}\log\sin\dfrac{2\pi l_i}{L}\right)
\]
is a Lyapunov-Reeb function on $\mathcal Q^{+}(C_m)$.
\end{lemma}
\begin{proof}
Clearly, $h$ is proper and lower-bounded. It remains to verify that $h$ has exactly one critical point. For $i=1,\cdots,m-1$, we introduce the variable substitution $\mu_i=l_i/L$, regard $h$ as a function of $(\mu,L):=(\mu_1,\cdots,\mu_{m-1},L)$, and derive
\[
h(\mu,L)=\log^2 L-\left(\sum\nolimits_{i=1}^{m-1}\log\sin2\pi\mu_i\right)
-\log\sin2\pi\left[1-\left(\sum\nolimits_{i=1}^{m-1}\mu_i\right)\right].
\]
A straightforward computation yields that $h$ has a unique critical point
\[
\big(\tilde{\mu},\tilde{L}\big):=\left(1/m,\cdots,1/m,1\right).
\]
Consequently, $h$ is a Lyapunov-Reeb function on $\mathcal Q^{+}(C_m)$.
\end{proof}

We are now in a position to prove Theorem~\ref{T-1-4} and Corollary~\ref{C-1-14}.

\begin{proof}[\textbf{Proof of Theorem~\ref{T-1-4}}]
We assume $m\geq 4$, as the case $m=3$ is trivial. Applying Theorems~\ref{T-1-7},~\ref{T-1-12},~\ref{T-4-16} and Lemma~\ref{L-4-17}, we reduce the proof to constructing a $\tau$-potential on $\mathcal M^{+}(C_m)$. As before, it is easy to verify that the function
\[
\Phi(p)=\sum_{\substack{[v_i,v_j]\in E \\ v_k\in V\setminus\{v_i,v_j\}}}\frac{1}{\left(\|p(v_i)-p(v_k)\|+\|p(v_j)-p(v_k)\|-\|p(v_i)-p(v_j)\|\right)^2}
\]
satisfies the required properties. This completes the proof.
\end{proof}

\begin{proof}[\textbf{Proof of Corollary~\ref{C-1-14}}]
The assertion follows from Theorem~\ref{T-1-7}.
\end{proof}

\subsection{Convexifying cycle linkages and configurations}To this end, let $\ell_i=\ell\big([v_i,v_{i+1}]\big)$ for $i=1,\cdots,m$ (with $v_{m+1}:=v_1$), and let
\[
\Xi=\left\{\theta=(\theta_2,\cdots,\theta_{m-1})\in (\mathbb R/2\pi\mathbb Z)^{m-2}:~u(\theta)=0\right\},
\]
where
\[
u(\theta):=\sqrt{\left(\ell_1
+\sum\nolimits_{j=2}^{m-1}\ell_j\cos\theta_j\right)^2
+\left(\sum\nolimits_{j=2}^{m-1}\ell_j\sin\theta_j\right)^2}-\ell_m.
\]
We now identify $\mathcal M^{+}(C_m,\ell)$ with an open subset $W_{m,\ell}\subset\Xi$ and regard
\[
f=\frac{1}{A}+w\Phi
\]
as a function on $W_{m,\ell}$. Consider the flow
\begin{equation}\label{E-4-7}
\frac{\mathrm d\theta_i}{\mathrm d t}=-\frac{\partial f}{\partial \theta_i}+\frac{\langle \nabla f, \nabla u\rangle}{\|\nabla u\|^2}\frac{\partial u}{\partial \theta_i},
\quad i=2,\cdots,m,
\end{equation}
which projects the negative gradient of  $f$  onto the tangent space of the constraint manifold $\Xi$. Following the proof of Lemma~\ref{L-4-1}, we show that
$\|\nabla u\|^2\neq 0$ on $W_{m,\ell}$, ensuring the flow~\eqref{E-4-7} is well defined. Let $\theta(t)$ be the solution of~\eqref{E-4-7}. A routine calculation yields
\[
\frac{\mathrm d u\big(\theta(t)\big)}{\mathrm d t}=0,
\]
so $\theta(t)$ remains on $\Xi$ whenever
$\theta(0)\in W_{m,\ell}\subset \Xi$. In addition,
\[
\frac{\mathrm d f\big(\theta(t)\big)}{\mathrm d t}\leq 0,
\]
from which we conclude that $\theta(t)$ is confined to $W_{m,\ell}$ for all $t\geq 0$. Let $\theta^\ast$ be the critical point of $f$. Because of Theorem~\ref{T-1-7}, the following result holds.
\begin{theorem}\label{T-4-18}
Suppose $m\geq 3$ and suppose $\ell:E\to\mathbb R_{+}$ satisfies
condition~$\mathrm{\mathbf{(c1)}}$. The flow~\eqref{E-4-7} starting at any point in $W_{m,\ell}$ exists for all $t\geq 0$ and converges to $\theta^\ast$ as $t\to+\infty$.
\end{theorem}

Let $p_t\in \mathcal M^{+}(C_m,\ell)$ denote the linkage parameterized by $\theta(t)$. Thus $\{p_t\}_{0\leq t\leq +\infty}$ yields a convexifying motion in
$\mathcal M^{+}(C_m,\ell)$ (see Fig.~\ref{F-5}). Analogously, by invoking the L-R function described in Theorem~\ref{T-4-16}, one can derive a convexifying motion in $\mathcal M^{+}(C_m)$.

\begin{figure}[htbp]
    \centering
    \includegraphics[width=0.66\textwidth]{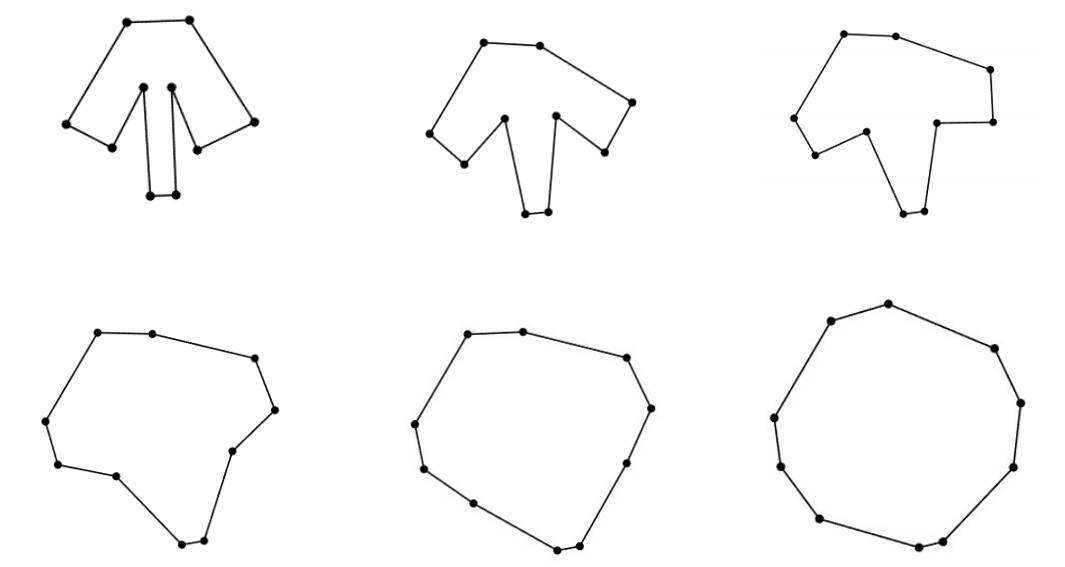}
    \caption{Convexifying a cycle linkage}
    \label{F-5}
\end{figure}

\section{The refolding problem}\label{S-5}

In fields such as computer graphics~\cite{Alt-2003} and robotic-arm folding~\cite{Lumelsky-1987}, it is often necessary to construct motions between distinct linkages (configurations). At first glance, one might deform a cyclic linkage into a cocircular shape and then fold it back into the shape of the second linkage. However, this approach results in an unnecessary detour, especially when the two linkages have ``near'' shapes. The refolding problem aims to find more direct motions between prescribed linkages (configurations). Iben, O'Brien, and Demaine~\cite{Demaine-2009} proposed a greedy algorithm for computing intersection-free planar refolding motions. Yet their approach lacks a well-defined criterion for assessing the quality of the resulting interpolation paths. To tackle this issue, we now construct complete Riemannian metrics on the moduli spaces. This requires several preliminary facts.

\begin{proposition}\label{P-5-1}
Let $(M,g)$ be a  Riemannian manifold and let $f:M\to\mathbb R$ be a proper smooth map. Then the following properties hold:
\begin{itemize}
\item[$(i)$] Let $g_0$ be the Euclidean metric on $\mathbb R$ and let $c_1>0$ be a constant. Then
$\tilde{g}=g+c_1f^{\ast} g_0$
is a complete Riemannian metric on $M$.
\item[$(ii)$] Suppose $d_0$ is the Euclidean distance on $\mathbb R$ and suppose $d_{g}$ is the distance induced by $g$. For any constant $c_2>0$, then
$\check{d}=\sqrt{d_{g}^2+c_2(f^{\ast}d_0)^2}$
defines a complete distance on $M$.
\end{itemize}
\end{proposition}
\begin{proof}
The first statement follows from the Hopf-Rinow Theorem and the properness of $f$.  Statement $(ii)$ is elementary with details omitted.
\end{proof}

We are in a position to show Theorem~\ref{T-1-16}.
\begin{proof}[\textbf{Proof of Theorem~\ref{T-1-16}}]
By virtue of Proposition~\ref{P-5-1}, the statement follows from Theorems~\ref{T-1-11},~\ref{T-1-12},~\ref{T-3-6}, and~\ref{T-4-16}.
\end{proof}

After equipping the moduli spaces with complete metrics, we reduce the refolding problem to computing geodesics joining pairs of points on the moduli spaces. A wide range of well-established numerical approaches are available for this purpose, including the shooting method~\cite{Noakes-1998}, fast marching method~\cite{Sethian-1996}, PDE-based techniques~\cite{Crane-Weischedel-Wardetzky}, among others.

Unfortunately, finding geodesics between two points in a general Riemannian manifold is typically computationally costly. This motivates us to formulate the following relaxed  task: \emph{Given a Riemannian manifold $(M,g)$ and a Lyapunov-Reeb function $f:M\to\mathbb R$, construct a ``reasonable'' curve connecting two prescribed points $p, q\in M$.} To achieve this objective, we fix a constant $c>0$ and introduce the following function
\[
\Psi_0(x)=d_{g}^2(x,q)+c |f(x)-f(q)|^2.
\]
It is clear that the map $\Psi_0:M\to \mathbb R$ is proper, lower-bounded, and attains its global minimum at $q$. By Proposition~\ref{P-5-1}, $\Psi_0(x)$ can actually be interpreted as the square of some distance between $x$ and $q$. We now formalize our algorithm for the aforementioned problem as follows:

\begin{itemize}
\item[$1.$] Consider the integral curve to the gradient flow system
\[
\begin{cases}
\displaystyle{\frac{\mathrm{d}x}{\mathrm{d}t}
=-\nabla \Psi_0(x)},\\
x(0)=p.
\end{cases}
\]
This curve is a geodesic emanating from $p$ and converging to some critical point $q'$ of $\Psi_0$. If $q'=q$, we directly obtain the desired curve.
\item[$2.$] In the case $q'\neq q$, we first connect $q$ to the critical point $q^\ast$ of $f$ via the negative gradient flow of $f$. Let $q_0=q,\,q_1,\,q_2,\cdots,q_n=q^\ast$ be a sequence of intermediate points along this flow trajectory. For $i=0,\cdots,n$, define
\[
\Psi_i(x)=d_{g}^2(x,q_i)+c|f(x)-f(q_i)|^2.
\]
Following the same argument as above, we consider the negative gradient flow $\gamma_1$ of $\Psi_1$ with initial datum $p$. If $\gamma_1$
converges to $q_1$, we construct the desired curve by concatenating $\gamma_1$ with the flow trajectory of $f$ from $q_1$ back to $q$. If $\gamma_1$ does not converge to $q_1$, we then examine the integral curves $\gamma_2,\cdots,\gamma_n$ generated by the negative flows of $\Psi_2,\cdots,\Psi_n$, all starting from $p$. If $\gamma_{i_0}$ converges to $q_{i_0}$ for some $i_0\in\{2,\cdots,n\}$, we likewise obtain the required curve.
\item[$3.$]If $\gamma_i$ fails to converge to $q_i$ for all $i=0,\cdots,n$, we can connect $p$ to $q^\ast$ via the negative gradient flow of $f$, and then connect $q^\ast$ to $q$ by reversing the flow trajectory from $q$ to $q^\ast$.
\end{itemize}

Note that the moduli spaces $\mathcal M(R_m)$, $\mathcal M^{+}(C_m)$, $\mathcal M(R_m,\ell)$, and $\mathcal M^{+}(C_m,\ell)$ all admit L-R functions. Using the above algorithm, we can then construct an animation\textemdash namely a continuous family of linkages and configurations\textemdash interpolating between two given images, which are modeled by polygonal linkages or configurations. Numerical experiments indicate that in most situations, with a suitable choice of the constant $c$, the desired curve is produced in the first step, meaning that geodesics (corresponding to ``optimal motions'') can often be constructed within our setting. See Fig.~\ref{F-6} for an illustration. The implementation code for the above-described algorithm is available at the following link:~\url{https://github.com/teba-hub/Refolding_Problem}.

\begin{figure}[htbp]
    \centering
    \includegraphics[width=0.75\textwidth]{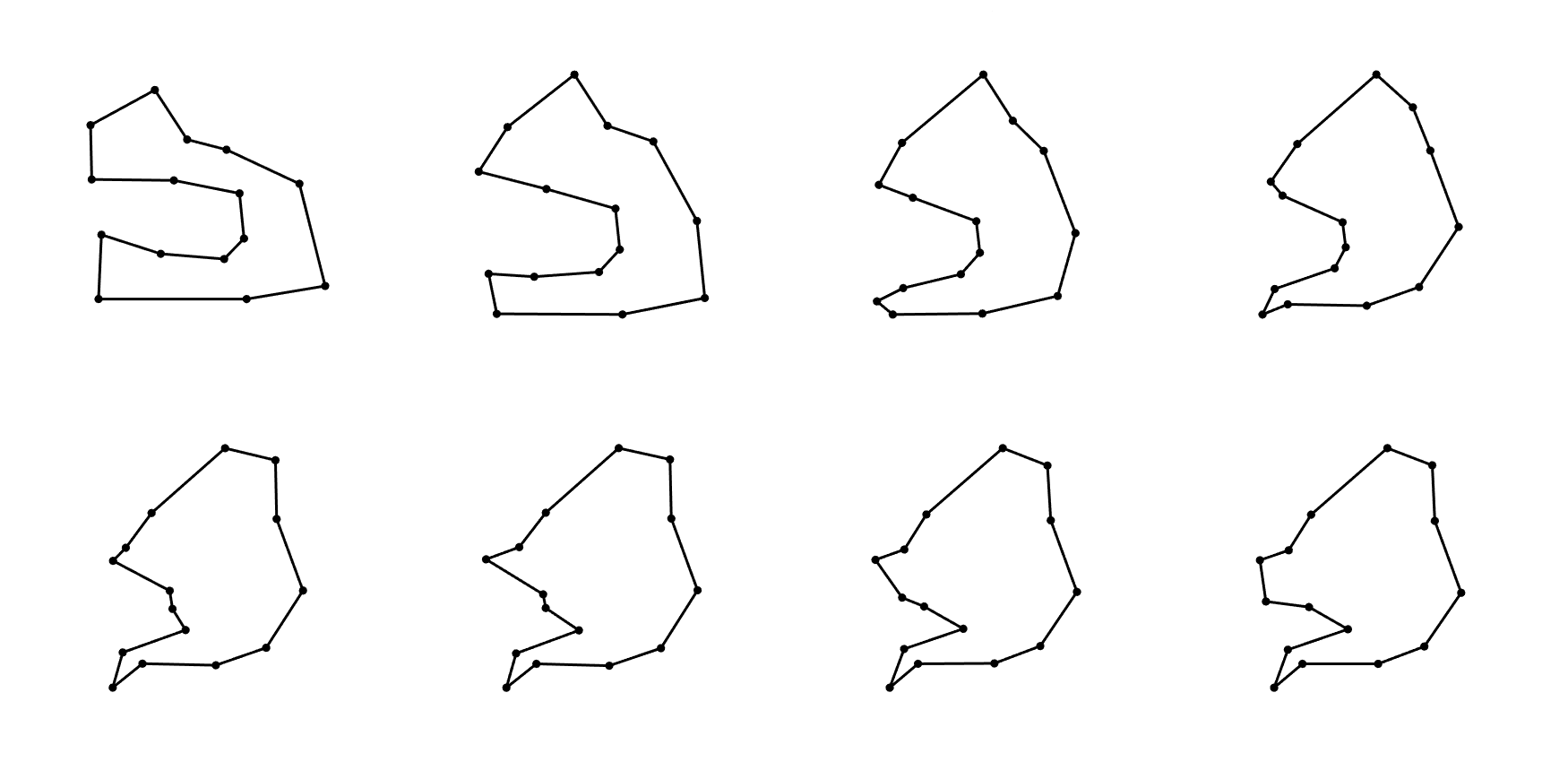}
    \caption{Continuous interpolation of two simple polygons}
    \label{F-6}
\end{figure}

\section{Questions}\label{S-6}
This paper leaves a number of questions open. We consider the following to be particularly interesting topics for further investigation.

1. \emph{Given $q\in \mathcal M^{+}(C_m)$, how to construct a Lyapunov-Reeb function on $\mathcal M^{+}(C_m)$ with $q$ as its critical point?} Analogous questions can also be posed for the moduli spaces $\mathcal M^{+}(C_m,\ell)$, $\mathcal M(R_m)$, and $\mathcal M(R_m,\ell)$. Note that the existence of such functions is trivial since these spaces are diffeomorphic to Euclidean spaces, but our goal is to derive explicit expressions for such functions. If this can be done, the associated negative gradient flows will yield more direct and efficient algorithms for the refolding problem.

2. Let $\mathcal M^{+}(C_\infty)$ represent the set of isometry classes of positively oriented, rectifiable simple closed curves in $\mathbb R^2$. \emph{Is it possible to construct a Lyapunov-Reeb function on $\mathcal M^{+}(C_\infty)$?} In view of Pardon's work~\cite{Pardon-2009}, $\mathcal M^{+}(C_\infty)$ is path-connected. Yet there is a need to uncover deeper topological properties of $\mathcal M^{+}(C_\infty)$ and devise practical algorithms for deforming rectifiable simple closed curves in $\mathbb R^2$. A natural approach to this end is constructing L-R functions on $\mathcal M^{+}(C_\infty)$.

3. Let $\mathcal M^{+}(C_m,\mathbb H^2)$ denote the set of isometry classes of positively oriented, simple $m$-gons in the hyperbolic plane $\mathbb H^2$. Similarly, we define the moduli spaces $\mathcal M^{+}(C_m,\ell, \mathbb H^2)$, $\mathcal M(R_m,\mathbb H^2)$, and $\mathcal M(R_m,\ell, \mathbb H^2)$. To capture their topological properties, we propose a conjecture analogous to Theorems~\ref{T-1-3} and~\ref{T-1-4}.
\begin{conjecture}
The following statements hold:
\begin{itemize}
\item[$(i)$]Let $m\geq 2$ and let $\ell:E\to\mathbb R_{+}$ be an assignment. Then $\mathcal M(R_m,\mathbb H^2)$ and $\mathcal M(R_m,\ell,\mathbb H^2)$ are path-connected and homeomorphic to Euclidean spaces.
\item[$(ii)$]Let $m\geq 3$ and let $\ell:E\to\mathbb R_{+}$ satisfy $\ell(e_i)<\sum_{e_j\in E\setminus\{e_i\}} \ell(e_j)$ for every edge $e_i\in E$. Then $\mathcal M^{+}(C_m,\mathbb H^2)$ and $\mathcal M^{+}(C_m,\ell, \mathbb H^2)$ are path-connected and homeomorphic to Euclidean spaces.
\end{itemize}
\end{conjecture}
A crucial step in addressing this conjecture is to establish the local expansive property for hyperbolic linkages, which requires deriving results parallel to Theorems~\ref{T-3-2} and~\ref{T-4-11}. This task is challenging, however, since the duality principle from linear programming does not readily generalize to hyperbolic geometry.

More generally, one may consider analogous questions in Hilbert geometry. Let $\Omega\subset\mathbb R^2$ be an \emph{open}, \emph{bounded}, and \emph{strictly convex} domain. For distinct points $x,y\in\Omega$, let the line through $x$ and $y$ intersect the boundary $\partial\Omega$ at $a,b$, ordered $a,x,y,b$ along this line. The Hilbert metric $d_\Omega$ on $\Omega$ is defined by
\[
d_\Omega(x,y)=\frac{1}{2}\log \frac{\|b-x\|\times \|y-a\|}{\|b-y\|\times\|x-a\|},
\]
where $d_\Omega(x,x)=0$. When $\Omega$ is an elliptical domain, $(\Omega,d_{\Omega})$ coincides with the Klein model of hyperbolic geometry, whereas on general strictly convex domains, Hilbert metrics are merely Finslerian. See the book edited by Papadopoulos and Troyanov~\cite{Papadopoulos-2014} for a survey of Hilbert geometry. Given a fixed point $x_\ast\in \Omega$, let $\mathcal M^{+}(C_m,\Omega,x_\ast,\mathbb R_{+})$ denote the collection of positively oriented, simple geodesic $m$-gons in $(\Omega,d_\Omega)$ satisfying the normalization conditions
\[
p(v_1)=x_\ast,\;p(v_2)=x_\ast+(\rho_1,0),
\]
where $\rho_1>0$. The moduli spaces $\mathcal M^{+}(C_m,\ell,\Omega,x_\ast,\mathbb R_{+})$, $\mathcal M(R_m,\Omega,,x_\ast,\mathbb R_{+})$, and $\mathcal M(R_m,\ell,\Omega,x_\ast,\mathbb R_{+})$ are defined analogously. %We formulate the following generalized conjecture:
\begin{conjecture}
These spaces satisfy the following characterizations:
\begin{itemize}
\item[$(i)$]Let $m\geq 2$ and let $\ell:E\to\mathbb R_{+}$ be an assignment. Then $\mathcal M(R_m,\Omega,x_\ast,\mathbb R_{+})$ and $\mathcal M(R_m,\ell,\Omega,x_\ast,\mathbb R_{+})$ are path-connected and homeomorphic to Euclidean spaces.
\item[$(ii)$]Let $m\geq 3$ and let $\ell:E\to\mathbb R_{+}$ satisfy $\ell(e_i)<\sum_{e_j\in E\setminus\{e_i\}} \ell(e_j)$ for every edge $e_i\in E$. Then $\mathcal M^{+}(C_m,\Omega,x_\ast,\mathbb R_{+})$ and $\mathcal M^{+}(C_m,\ell,\Omega,x_\ast,\mathbb R_{+})$ are path-connected and homeomorphic to Euclidean spaces.
\end{itemize}
\end{conjecture}
Aside from a few numerical experiments, no known results or prior ideas exist for this conjecture. It is thus desirable to establish tools relating Hilbert geometries associated with different convex domains.

\section{Appendix}\label{S-7}

In this section, we provide the proof of Lemma~\ref{L-4-5}.

\begin{proof}[\textbf{Proof of Lemma~\ref{L-4-5}}]
Consider the expression for the area
\[
A(\Delta)=\frac{1}{2}l_jl_k\sin \alpha_i.
\]
Taking the derivative with respect to $l_i$ yields
\begin{equation}\label{E-7-1}
\frac{\partial A(\Delta)}{\partial l_i}=\frac{1}{2}l_jl_k\cos \alpha_i\frac{\partial \alpha_i}{\partial l_i}.
\end{equation}
Meanwhile, from the Law of Cosines
\[
\cos \alpha_i=\frac{l_j^2+l_k^2-l_i^2}{2l_jl_k},
\]
it follows that
\begin{equation}\label{E-7-2}
\frac{\partial \alpha_i}{\partial l_i}=\frac{l_i}{l_jl_k\sin \alpha_i}.
\end{equation}
Combining~\eqref{E-7-1} and~\eqref{E-7-2}, we have
\begin{equation}\label{E-7-3}
\frac{\partial A(\Delta)}{\partial l_i}=\frac{l_i\cot\alpha_i}{2},
\end{equation}
which concludes assertion $(i)$.

For assertion $(ii)$, differentiating both sides of~\eqref{E-7-3} with respect to $l_i$ gives
\[
\frac{\partial A^2(\Delta)}{\partial l_i^2}=\frac{\cot\alpha_i}{2}-\frac{l_i}{2\sin^2 \alpha_i}\frac{\partial \alpha_i}{\partial l_i}
\]
Substituting~\eqref{E-7-2} into the above expression, we derive
\[
\frac{\partial^2 A(\Delta)}{\partial l_i^2}=\frac{\cot\alpha_i}{2}
    -\frac{1}{2l_il_jl_k}\left(\frac{l_i}{\sin\alpha_i}\right)^3.
\]

It remains to consider assertion $(iii)$. By virtue of~\eqref{E-7-3}, we infer
\[
\frac{\partial^2 A(\Delta)}{\partial l_i\partial l_j}=-\frac{l_i}{2\sin^2 \alpha_i}\frac{\partial \alpha_i}{\partial l_j}.
\]
As in previous steps, the Law of Cosines yields
\[
\frac{\partial \alpha_i}{\partial l_j}=-\frac{l_i\cos\alpha_k}{l_jl_k\sin \alpha_i}.
\]
Combining the above two formulas, we obtain
\[
\frac{\partial A^2(\Delta)}{\partial l_i\partial l_j}=\frac{1}{2l_il_jl_k}\left(\frac{l_i}{\sin \alpha_i}\right)^3\cos\alpha_k.
\]
This completes the proof.
\end{proof}

\section{Statements and Declarations}
%\begin{itemize}
%\item Funding.
%\item Conflicts of interest/Complete interests.
%\end{itemize}
%\subsection*{\textbf{Funding}}
%\bigskip
\noindent\textbf{Funding.}
Te Ba and Ze Zhou were supported by NSFC Grant No. 12422105 and Grant No. 12371075.

%This work was supported by the National Natural Science Foundation of China (Grant No. 12371075).

\medskip
\noindent\textbf{Conflict of interest/Competing interests.}
The authors have no relevant financial or non-financial interests to disclose.

\medskip
\noindent\textbf{Ethics approval and consent to participate.} Not Applicable.

%\medskip
%\noindent\textbf{Consent for publication.} Not Applicable.

\medskip
\noindent\textbf{Data availability.} Not Applicable.
%All figures in this paper originate from the authors and have not been published anywhere else. Apart from the figures, no other data is used in the manuscript.

\medskip
\noindent\textbf{Materials availability.} Not Applicable.

\medskip
\noindent\textbf{Code availability.} The code used to generate some figures in this paper, along with all other repository contents provided by the authors, is publicly available at \url{https://github.com/teba-hub/Refolding_Problem}.

\medskip
\noindent\textbf{Use of AI.} AI tools were used only for text editing and English language polishing. All research work, original writing and conclusions were completed by the authors. The authors reviewed all AI outputs and take full responsibility for this manuscript.

\medskip
\noindent\textbf{Author contribution.} Both authors contributed equally to this work.

%\section*{Acknowledgments}
%Te Ba and Ze Zhou were supported by NSFC Grant No. 12422105 and Grant No. 12371075.

\end{document}